\documentclass{article}

\usepackage[utf8]{inputenc}
\usepackage[T1]{fontenc}
\usepackage{amsmath}

\usepackage[amsthm]{ntheorem}
\usepackage{amsfonts}
\usepackage{amssymb}
\usepackage{array}
\usepackage{dsfont}
\usepackage{hyperref}
\usepackage{enumerate}
\usepackage[all]{xy}
\usepackage{pgf,tikz}
\usepackage[caption=false]{subfig}

\usepackage{algorithm}
\usepackage{algorithmic}

\def\tp{\ensuremath{\hdots}}
\def\infini{\ensuremath{\infty}}
\def\fc{\ensuremath{\mathds 1}}

\def\Bc{\ensuremath{\mathcal B}}

\def\Sc{\ensuremath{\mathcal S}}

\def\E{\ensuremath{\mathbb E}}

\def\N{\ensuremath{\mathbb N}}
\def\P{\ensuremath{\mathbb P}}

\def\R{\ensuremath{\mathbb R}}

\def\Z{\ensuremath{\mathbb Z}}

\usepackage{pgf,tikz}
\usetikzlibrary{arrows}
\usetikzlibrary[patterns]
\definecolor{cqcqcq}{rgb}{0.7529411764705882,0.7529411764705882,0.7529411764705882}
\definecolor{ffqqqq}{rgb}{1.,0.,0.}
\definecolor{ffqqtt}{rgb}{1.,0.,0.2}
\definecolor{qqwuqq}{rgb}{0.,0.39215686274509803,0.}
\definecolor{qqqqff}{rgb}{0.,0.,1.}
\definecolor{uuuuuu}{rgb}{0.26666666666666666,0.26666666666666666,0.26666666666666666}
\definecolor{xdxdff}{rgb}{0.49019607843137253,0.49019607843137253,1.}
\definecolor{uququq}{rgb}{0.25,0.25,0.25}

\def\aM{\text{argmax}}
\def\am{\text{argmin}}

\usepackage{cleveref}

\newtheorem{theo}{Theorem}[section]

\newtheorem{Theorem}[theo]{Theorem}
\newtheorem{Proposition}[theo]{Proposition}

\newtheorem{Corollary}[theo]{Corollary}
\newtheorem{Remark}[theo]{Remark}

\newtheorem{Definition}[theo]{Definition}

\newtheorem{Lemma}[theo]{Lemma}
\newtheorem{Example}[theo]{Example}

\newcommand{\psh}[2]{\ensuremath{\left\langle #1,#2\right\rangle}}

\title{Inconsistency of Template Estimation by Minimizing of the Variance/Pre-Variance in the Quotient Space $^\dagger$}

\author{
Loïc Devilliers\footnote{Université C\^ote d’Azur, Inria, France, \url{loic.devilliers@inria.fr}},
Stéphanie Allassonnière\footnote{CMAP, Ecole polytechnique, CNRS, Université Paris-Saclay, 91128, Palaiseau, France},
Alain Trouvé\footnote{CMLA, ENS Cachan, CNRS, Université Paris-Saclay, 94235 Cachan, France},\\
and Xavier Pennec\footnote{Université C\^ote d’Azur, Inria, France}}

\begin{document}
\maketitle

\abstract{
We tackle the problem of template estimation when data have been randomly deformed under a group action in the presence of noise. In order to estimate the template, one often minimizes the variance when the influence of the transformations have been removed (computation of the Fréchet mean in the quotient space). 
The consistency bias is defined as the distance (possibly zero) between the orbit of the template and the orbit of one element which minimizes the variance. 
In the first part, we restrict ourselves to isometric group action, in this case the Hilbertian distance is invariant under the group action. We establish an asymptotic behavior of the consistency bias which is linear with respect to the noise level. As a result the inconsistency is unavoidable as soon as the noise is enough. In practice, template estimation with a finite sample is often done with an algorithm called "max-max". 
In the second part, also in the case of isometric group finite, we show the convergence of this algorithm to an empirical Karcher mean.
Our numerical experiments show that the bias observed in practice can not be attributed to the small sample size or to a convergence problem but is indeed due to the previously studied inconsistency.
In a third part, we also present some insights of the case of a non invariant distance with respect to the group action. We will see that the inconsistency still holds as soon as the noise level is large enough. Moreover we prove the inconsistency even when a regularization term is added.}
\newpage

\tableofcontents

\section{Introduction}

\subsection{General Introduction}
Template estimation is a well known issue in different fields such as statistics on signals~\cite{kur}, shape theory, computational anatomy~\cite{gui,jos,coo} etc. In these fields, the template (which can be viewed as the prototype of our data) can be (according to different vocabulary) shifted, transformed, wrapped or deformed due to different groups acting on data. Moreover, due to a limited precision in the measurement, the presence of noise is almost always unavoidable. These mixed effects on data lead us to study the consistency of algorithms which claim to compute the template. A popular algorithm consists in the minimization of the variance, in other words, the computation of the Fréchet mean in quotient space. This method has been already proved to be inconsistent~\cite{big,mio2,dev2}. 
In~\cite{big} the authors proves the inconsistency with a lower bound of the expectation of the error between the original template and the estimated template with a finite sample, they deduce that this expectation does not go to zero as the size of the sample goes to infinity. This work was done in a functional space, where functions only observed at a finite number of points of the functions were observed. In this case one can model these observable values on a grid. When the resolution of the grid goes to zero, one can show the consistency~\cite{pan} by using the Fréchet mean with the Wasserstein distance on the space of measures rather than in the space of functions. However, in (medical) images the number of pixels or voxels is~finite.

In~\cite{mio2}, the authors demonstrated the inconsistency in a finite dimensional manifold with Gaussian noise, when the noisel level tends to zero. In our previous work~\cite{dev2}, we focused our study on the inconsistency with Hilbert Space (including infinite dimensional case) as ambient space. This current paper is an extension of a conference paper~\cite{dev3}.

{\subsection{Why Using a Group Action? Comparison with the Standard Norm}}
{
In the following, we take a simple example which justifies the use of the group action in order to compare the shape of two functions:

On Figure \ref{fig:comp}, suppose that you want to compare these functions. The simplest way to compare $f_0$ with $f_1$ would be to compute the $L^2$-norm (or any other norm) of $f_0$ $-$ $f_1$, if we do that we have that $\|f_0-f_1\|\simeq 0.6$. Likewise $\|f_0-f_2\|\simeq 0.6$, therefore the norm tells us that $f_0$ is at the same distance from $f_1$ and from $f_2$. Yet, our eyes would say that $f_0$, $f_1$ have the same shape, contrarily to $f_0$ and $f_2$. Therefore the simple use of the $L^2$-norm in the space of functions is not enough. To have a relevant way to compare functions, one can register functions first. Firstly, we estimate the better time translation which aligns $f_0$ and $f_1$ and secondly, we compute the $L^2$-norm after this alignment step. On this example, we find that the distance is now $\simeq$0.02. On the contrarily, after alignment the distance between $f_0$ and $f_2$ is still $\simeq$0.6. With this new way of comparing functions, the functions $f_0$ looks like $f_1$ but do not look like $f_2$.
This fits with our intuition. That is why we use a group action in order to perform statistics. In the following paragraph, we precise how to do it in general.}

This idea of using deformations/transformation in order to compare things is not new. It was already proposed by Darcy Thompson~\cite{dar} in the beginning of the 20th century, in order to classify species.

\begin{figure}[h]
\centering
\includegraphics[clip=true, trim=2cm 7cm 2cm 7cm, width=0.5\textwidth]{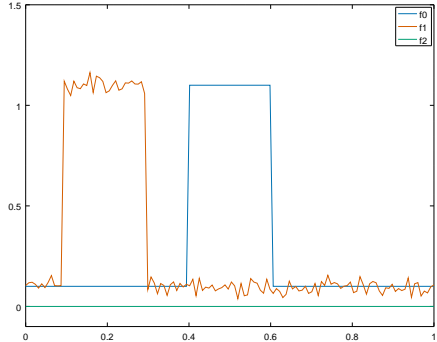}
\caption{Three functions defined on the interval $[0,1]$. The blue one ($f_0$) is a step function, the red one ($f_1$) is a translated version of the blue one when noise has been added, and the green one ($f_3$) is the null function.}
\label{fig:comp}
\end{figure}

\subsection{Settings and Notation}
In this paper, we suppose that observations belong to a Hilbert space $(M,\: \psh{\cdot}{\cdot})$, we denote by $\|\cdot\|$ the norm associated to the dot product $\psh{\cdot}{\cdot}$. We also consider a group of transformation $G$ which acts on $M$ the space of observations. This means that $g'\cdot (g\cdot x)=(g'g)\cdot x$ and $e\cdot x=x$ for all $x\in M$, $g,\: g'\in G$, where $e$ is the identity element of $G$. Note that in this article, $g\cdot x$ is the result of the action of $g$ on $x$, and $\cdot$ should not to be confused with the multiplication of real numbers noted $\times$

{\em The generative model}  
is the following: we transform an unknown template~$t_0\in M$ with $\Phi$ a random and unknown element of the group $G$ and we add some noise. Let $\sigma$ be a positive noise level and $\epsilon$ a standardized noise: $\E(\epsilon)=0$, $\E(\|\epsilon\|^2)=1$. Moreover we suppose that $\epsilon$ and $\Phi$ are independent random variables. Finally, the only observable random variable is: 
\begin{equation}Y=\Phi\cdot t_0+\sigma\epsilon\label{modgenforward}.\end{equation} 

This generative model is commonly used in Computational anatomy in diverse frameworks, for instance
with currents~\cite{dur,gla}, {varifolds~\cite{cha}, LDDMM on images~\cite{mil} but also in functional data analysis~\cite{kur}. 
All these works are applied in different spaces, for instance, the varifold builds an embedding of the surfaces into an Hilbert space, and a group of diffeomorphisms have the ability of deform these surfaces. 
Supposing a general group action on a space with the generative model~\eqref{modgenforward} allows us to embed all these various situations into one abstract model, and to study template estimation in this abstract model.
}

Example of noise: if we assume that the noise is independent and identically distributed on each pixel or voxel with a standard deviation $w$, then $\sigma=\sqrt{N}w$, where $N$ is the number of pixels/voxels. However, the noise which we consider can be more general: we do not require the fact that the noise is independent over each region of the space $M$.

Note that the inconsistency of Template estimation can be also studied with an alternative generative model, called backward model where $Y=\Phi\cdot (t_0+\sigma \epsilon)$~\cite{dev2}. Some authors also use the term \textit{{perturbation model}} see~\cite{huc,roh,goo}.  

{\em Quotient space:} the random transformation of the template by the group leads us to project the observation $Y$ into the quotient space. The quotient space is defined as the set containing all the orbit $[x]=\{g\cdot x,\: g\in G\}$ for $x\in M$. 
The set which is constituted of all orbits is call the quotient space $M$ by the group $G$ and is noted by:
\begin{equation*}
Q=M/G=\{[x],\: x\in M\}.
\end{equation*}

As we want to do statistics on this space, we aim to equip the quotient with a metric. One often requires that $d_M$ the distance in the ambient space is invariant under the group action $G$, this means that
\begin{equation*}
\forall m,n\in M,\: \forall g\in G\quad  d_M(g\cdot m,g\cdot n)=d_M(m,n).
\end{equation*}

If $d_M$ is invariant and if the orbits are closed sets (if the orbits are not closed sets, it is possible to have $d_Q([a],[b])=0$ even if $[a]\neq [b]$, in this case we call $d_Q$ a pseudo-distance. Nevertheless, this has no consequence in this paper if $d_Q$ is only a pseudo-distance), then 
\begin{equation*}
d_Q([x],[y])=\underset{g\in G}{\inf} d_M(x,g\cdot y),
\end{equation*}
is well defined, and $d_Q$ is a distance in the quotient space.
The quotient distance $d_Q([x],[y])$ is the distance between $x$ and $y'$ where $y'$ is the registration of $y$ with respect to $x$. We say in this case that $y'$ is in optimal position with respect to $x$.

One particular distance in the ambient space $M$, which we use in all this article, is the distance given by the norm of the Hilbert space: $d_M(a,b)=\|a-b\|$. 
Moreover we say that $G$ acts isometrically on $M$, if $x\mapsto g\cdot x$ is a linear map which leaves the norm unchanged. In this case $d_M$ the distance given by the norm of the Hilbert space is invariant under the group action.
The quotient (pseudo)-distance is, in this case (see \cref{dq2}), $d_Q([a],[b])=\underset{g\in G}{\inf} \|a-g\cdot b\|$.

\begin{figure}[h]
\centering
\begin{tikzpicture}[scale=1.1]
\draw (0,0) node {$\bullet$} node[below] {$0$};
\draw (0,1) node {$\bullet$} ;
\draw (-2,0) node {$\bullet$} ;
\draw (0,0) circle (1) ;
\draw (0,0) circle (2) ;
\draw (0,1) node[above]{$p=(0,1)$} ;
\draw (-2,0) node[left]{$q=(-2,0)$} ;
\draw[red,thick,<->] (1,0)--(2,0);
\draw[red] (2,0) node[right] {$d_Q([p],[q])=1$};
\end{tikzpicture}
\caption[Due to the invariant action, the orbits are parallel]{Due to the invariant action, the orbits are parallel. Here the orbits are circles centred at $0$. This~is the case when the group $G$ is the group of rotations.}
\label{dq2} 
\end{figure}
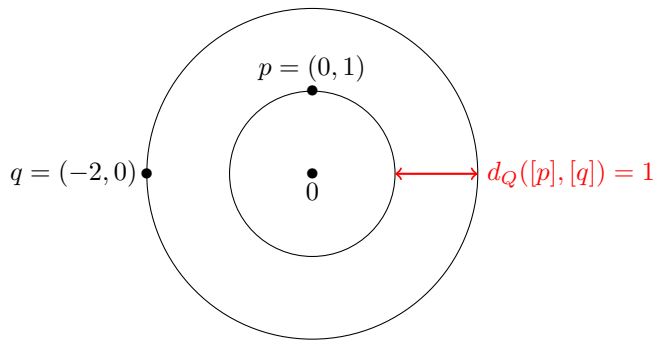

\begin{Remark}
When $G$ acts isometrically on $M$ a Hilbert space, by expansion of the squared norm we have:
\begin{equation*}
d_Q([a],[b])^2=\|a\|^2-2\underset{g\in G}{\sup} \psh{a}{g\cdot b}+\|b\|^2
\end{equation*}

Thus, even if the quotient space is not a linear space, we have a ``polarization identity'' in the quotient space:
\begin{equation}
\underset{g\in G}{\sup} \psh{a}{g\cdot b}=\frac12\left(\|a\|^2+\|b\|^2-d^2_Q([a],[b]\right)=\frac12\left(d_Q^2([a],[0])+d_Q^2([b],[0])-d_Q^2([a],[b]\right)
\label{polari}
\end{equation}
\end{Remark}

When the distance given by the norm is invariant under the group action, we define the variance of the random orbit~$[Y]$ as the expectation of the (pseudo)-distance between the random orbit~$[Y]$ and the orbit of a point~$x$ in $M$:
\begin{equation*}
F(x)=\E(d_Q^2([x],[Y]))=\E(\inf_{g\in G}\|g\cdot x-Y\|^2)=\E(\inf_{g\in G}\|x-g\cdot Y\|^2).
\end{equation*}

Note that $F(x)$ is well defined for all $x\in M$ because $\E(\|Y\|^2)$ is finite. 
Moreover, since \mbox{$F(g\cdot x)=F(x)$}, for all $x\in M$ and $g\in G$, the variance $F$ is well defined in the quotient space: $[x]\mapsto F(x)$ does have a sense.

{
Moreover, in presence of a sample of the observable variable $Y$ noted $Y_1,\tp, Y_n$, one can define the empirical variance of a point~$x$ in $M$ : 
\begin{equation*}
F_n(x)=\sum_{k=1}^{n} (\inf_{g\in G}\|g\cdot x-Y_i\|^2)=\sum_{k=1}^n(\inf_{g\in G}\|x-g\cdot Y_i\|^2).
\end{equation*}
\begin{Definition}
Template estimation is performed by minimizing $F_n$ :
\begin{equation*}
\widehat{t_0}_n=\underset{x\in M}{\am} \ F_n,
\end{equation*}
\end{Definition}
In order to study this estimation method, one can look the limit of this estimator when the number of data $n$ tends to $+\infini$, in this case, the estimation becomes:
\begin{equation*}
\widehat{t_0}_{\infini} =\underset{x\in M}{\am}\ F
\end{equation*}
}

If $m_\star\in H$ minimizes $F$, then $[m_\star]$ is called a Fréchet mean of $[Y]$.
\begin{Definition} {We say that the estimation is consistent if $t_0$ minimizes $F$.}
Moreover the consistency bias, noted $CB$, is the (pseudo)-distance between the orbit of the template $[t_0]$ and $[m_\star]$: $CB=d_Q([t_0],[m_\star])$. If such a $m_\star$ does not exist, then the consistency bias is infinite.
\end{Definition}
Note that, if the action is not isometric and is not either invariant, a priori $d_Q$ is no longer a (pseudo)-distance in the quotient space (this point is discussed in Section \ref{sec:noni}). However one can still define $F$ and wonder if the minimization of $F$ is a consistent estimator of $t_0$. In this case, we call $F$ a~pre-variance.

\subsection{Questions and Contributions}
This setting leads us to wonder about few things listed below:\\

{\em Questions}: 
\begin{itemize}
\item Is $t_0$ a minimum of the variance or the pre-variance?
\item What is the behavior of the consistency bias with respect to the noise level?
\item How to perform such a minimization of the variance? Indeed, in practice we have only a sample and not the whole distribution.
\end{itemize}

{\em Contribution}: In the case of an isometric action, we provide a Taylor expansion of the consistency bias when the noise level $\sigma$ tends to infinity. As we do not have the whole distribution, we minimize the empirical variance given a sample. An element which minimizes this empirical variance is called an empirical Fréchet mean. We already know that the empirical Fréchet mean converges to the Fréchet mean when the sample size tends to infinity~\cite{zie}. Therefore our problem is reduced to finding an empirical Fréchet mean with a finite but sufficiently large sample. One algorithm called the ``max-max'' algorithm~\cite{all} aims to compute such an empirical Fréchet mean.
We establish some properties of the convergence of this algorithm. In particular, when the group is finite, the algorithm converges in a finite number of steps to an empirical Karcher mean (a local minimum of the empirical variance given a sample). This helps us to illustrate the inconsistency in this very simple framework. 

We would like to insist on this point: the noise is created in the ambient space with our generative model and the computation of the Fréchet mean is done in the quotient space, this interaction induces an inconsistency. On the opposite, if one models the noise directly in the quotient space and compute the Fréchet mean in the quotient space, we have no reason to suspect any inconsistency.

Moreover it is also possible to define and use isometric actions on curves~\cite{hit,kur} or on surfaces~\cite{kur2} where our work can be directly applied. The previous works related to the inconsistency of template estimation~\cite{big,mio2,dev2} focused on isometric action, which is a restriction to real applications. That~is why we provide, in  Section \ref{sec:noni}, some insights of the non invariant case: the inconsistency also appears as soon as the noise level is large enough.

This article is organized as follows: 
\Cref{sec:iso} is dedicated for isometric action. More precisely, in \Cref{sec:inc}, we study the presence of the inconsistency and we establish the asymptotic behavior when the noise parameter $\sigma$ tends to $\infini$. In \Cref{sec:max} we detail the max-max algorithm and its properties. In \Cref{sec:sim} we illustrate the inconsistency with synthetic data. Finally in \Cref{sec:noni}, we prove the inconsistency for more general group action, when the noise level is large enough.
We do it in two settings, the first one is that the group contains a subgroup acting isometrically on $M$, the second one is that the group acts linearly on the space $M$.

\section{Inconsistency of Template Estimation with an Isometric Action}
\label{sec:iso}

\subsection{Congruent Section and Computation of Fréchet Mean in Quotient Space}
\label{subsec:equisec}

Given points $m$ and $y$, there is a priori no closed formed expression in order to compute the quotient distance $\underset{g\in G}{\inf}\|g\cdot m-y\|$. Therefore computing and minimizing the variance in the quotient does not seem straightforward. 
There is one case where it may be possible: the existence of a congruent section. We say that $s:Q\to M$ is a section if $\pi\circ s=Id$, where $\pi:M\to Q$ is the canonical projection into the quotient space. Moreover we say that the section $s$ is congruent if:
\begin{equation*}
\forall o,o'\in Q \quad \|s(o)-s(o')\|=d_Q(o,o').
\end{equation*}

Then $\Sc=s(Q)$ the image of the quotient by the section is a part of $M$ which has an interesting~property:
\begin{equation*}
\forall p,q\in \Sc, \|p-q\|=d_Q([p],[q]).
\end{equation*}

In other words, the section gives us a part of $M$ containing a point of each orbit such that all points in $\Sc$ are already registered. Moreover, if $s$ is a section, $s':[m]\mapsto g\cdot s([m])$ is also a section, without loss of generality we can assume that $t_0=s([t_0])$.

In this case, the variance is equal to:
\begin{equation*}
F(m)=\E(\|s([m])-s([Y])\|^2),
\end{equation*}
where we recognize the variance of the random variable $s([Y])$. As we know that the element which minimizes the variance in a linear space is given by the expected value, we have that:

\begin{equation*}
F(m)\geq F(\E(s([Y]))).
\end{equation*}

Moreover this inequality is strict if and only if $m$ and $\E(s([Y]))$ are not in the same orbit. 

Therefore, we have a method in order to know if the estimation is consistent or not: computing $\E(s([Y]))$ and verifying if $t_0$ and $\E(s([Y]))$ are in the same orbit, and the consistency bias is given by $d_Q([t_0],[\E(s([Y]))])$.
Moreover if we take $m\in \Sc$, we have $F(m)=\E(\|m-s([Y])\|^2)$ and it is now straightforward that $F|_\Sc$ the restriction of $F$ to $\Sc$ is differentiable on $\Sc$ (We say that $F|_\Sc$ is differentiable on $\Sc$, even if $\Sc$ is not open, because $m\mapsto \E(\|m-s([Y])\|^2)$ is defined and differentiable on $M$, and is equal to $F|\Sc$), and that $\nabla F|_\Sc(m)=m-\E(s([Y])$ in particular $\|\nabla F|_\Sc(t_0)\|=\|t_0-\E(s([Y]))\|$ gives us the value of the bias. 

\begin{Example}
The action of rotations: $G=SO(n)$ acts isometrically on $M=\R^n$. We notice that the quotient distance is $d_Q([x],[y])=|\|x\|-\|y\||$.
We can check that $s([x])=\|x\|v$ is a section for $v$ an unitary vector. Therefore the computation of the bias is given by $d_Q([t_0],[\E(s([Y])])=|\E(\|Y\|)-\|t_0\|)|$.
\end{Example}
Unfortunately, the congruent section generally does not exist. Let us give an example:
\begin{Example}
Taking $N\in \N$ with $N\geq 3$, we consider the action of $G=\Z/N\Z$ on $M=\R^N$ by time translation, for $\bar{k}\in \Z/N\Z$, and $(x_1,x_2,\tp,x_N)$:
\begin{equation*}
\bar{k}\cdot (x_1,x_2,\tp,x_N)=(x_{1+k},x_{2+k},\tp ,x_{N+k}),
\end{equation*}
where indexes are taken modulo $N$. 
If we take $p_1=(0,5,0,\tp,0)$, $p_2=(0,3,2,0,\tp,0)$, $p_3=(2,3,0,\tp,0)$. By hand we can check that there is no $x\in [p_1]$, $y\in [p_2]$ and $z\in [p_3]$ such that $\|x-y\|=d_Q([p_1],[p_2])$, $\|x-z\|=d_Q([p_1],[p_3])$, and $\|y-z\|=d_Q([p_2],[p_3])$. Thus, a congruent section in $Q=M/G$ does not exists.
\end{Example}
We can generalize this simple example by taking a non finite group:

\begin{Example}
Let us take $M=L^2(\R/\Z)$ the set of 1-periodic functions such that $\int_0^1 f^2(t)dt<+\infini$. $G=\R/\Z$ acts on $L^2(\R/\Z)$ by time translation defined by:

\begin{equation*}
\tau\in \R/\Z, \: f\in L^2(\R/\Z)\mapsto f_\tau \mbox{ with} f(x)=f(x+\tau).
\end{equation*}
Then a section in $Q=M/G$ does not exists.
\end{Example}
\begin{proof}
Let us take $f_1=\fc_{[\frac14,\frac34]}$, $f_2=f_1+2\fc_{[\frac14,\frac14+\eta]}$ and~$f_2=f_1+2\fc_{[\frac14+\eta,\frac14+2\eta]}$ for some $\eta\in (0,\frac14)$ (see \cref{fig:f1f2f3}). Let~us suppose that a section $s$ exists, then without loss of generality we can assume that $s([f_1])=f_1$, then we should have $\|f_1-s([f_2])\|=\|s([f_1])-s([f_2])\|=d_Q([f_1],[f_2])$ in other words, $s([f_2])$ should be registered with respect to $f_1$. 
For $\tau\in \R/\Z$ we can verify that $\|f_1-\tau\cdot f_2\|\geq \|f_1-f_2\|$ and that this inequality is strict as soon as $\tau\neq 0$. Then $f_2$ is the only element of $[f_2]$ registered with $f_1$ then $s([f_2])=f_2$. Likewise for $s([f_3])=f_3$, then we should have:
\begin{equation*}
d_Q([f_2],[f_3])=\|f_2-f_3\|,
\end{equation*}

However it is easy to verify that \mbox{$d_Q^2([f_2],[f_3])\leq\| \eta\cdot f_2-  f_3\|^2=2\eta<8\eta=\|f_2-f_3\|^2=d_Q([f_2],[f_3])$}. This is a contradiction. Therefore, a congruent section does not exist.
\end{proof}
\vspace{-12pt}
\begin{figure}[h]
\centering
\subfloat[$f_1$]{\includegraphics[clip=true,trim=2.6cm 7.5cm 2.55cm 7cm, width=0.315\textwidth]{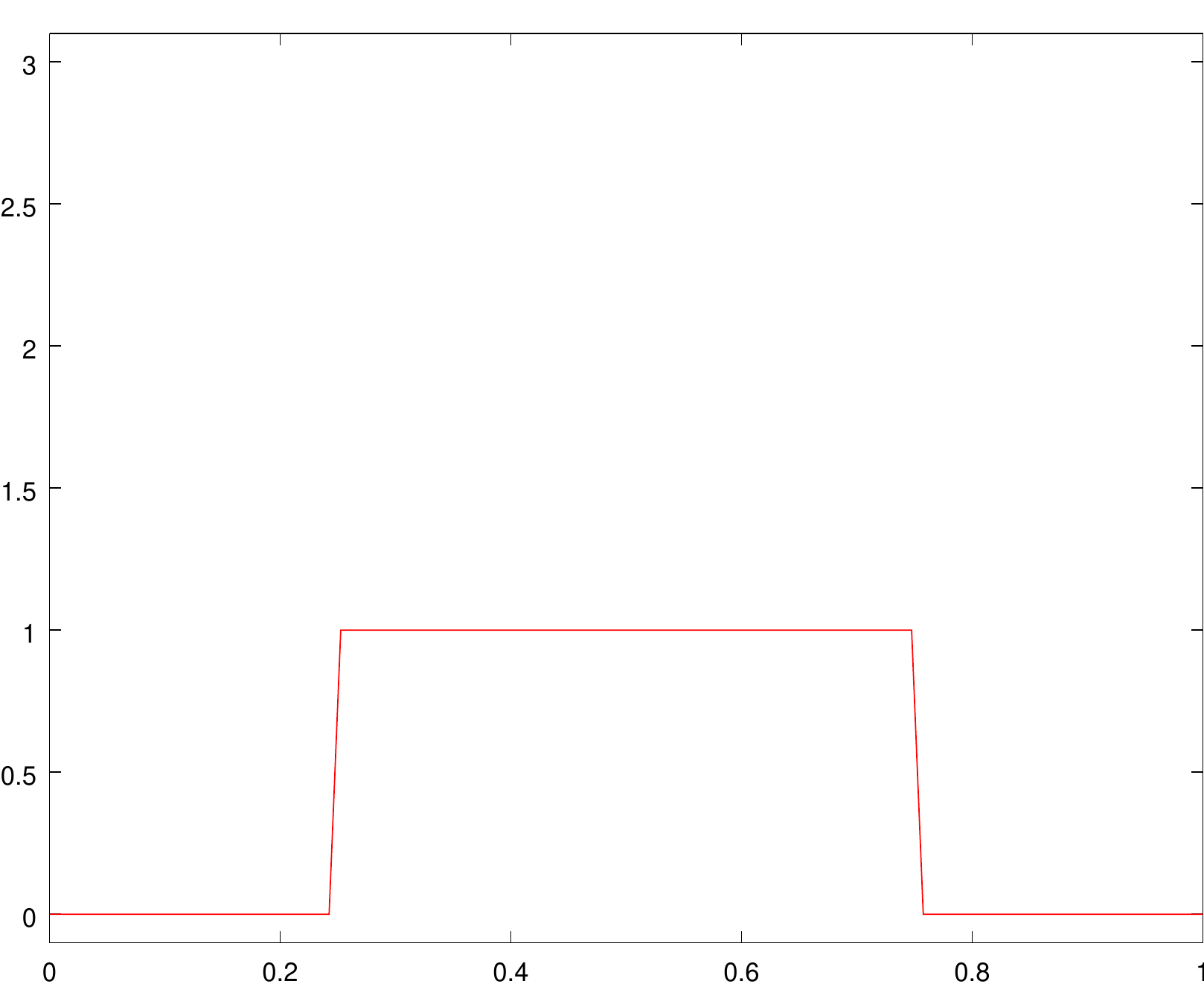}}
\:
\subfloat[$f_2$]{\includegraphics[clip=true,trim=2.6cm 7.5cm 2.55cm 7cm, width=0.315\textwidth]{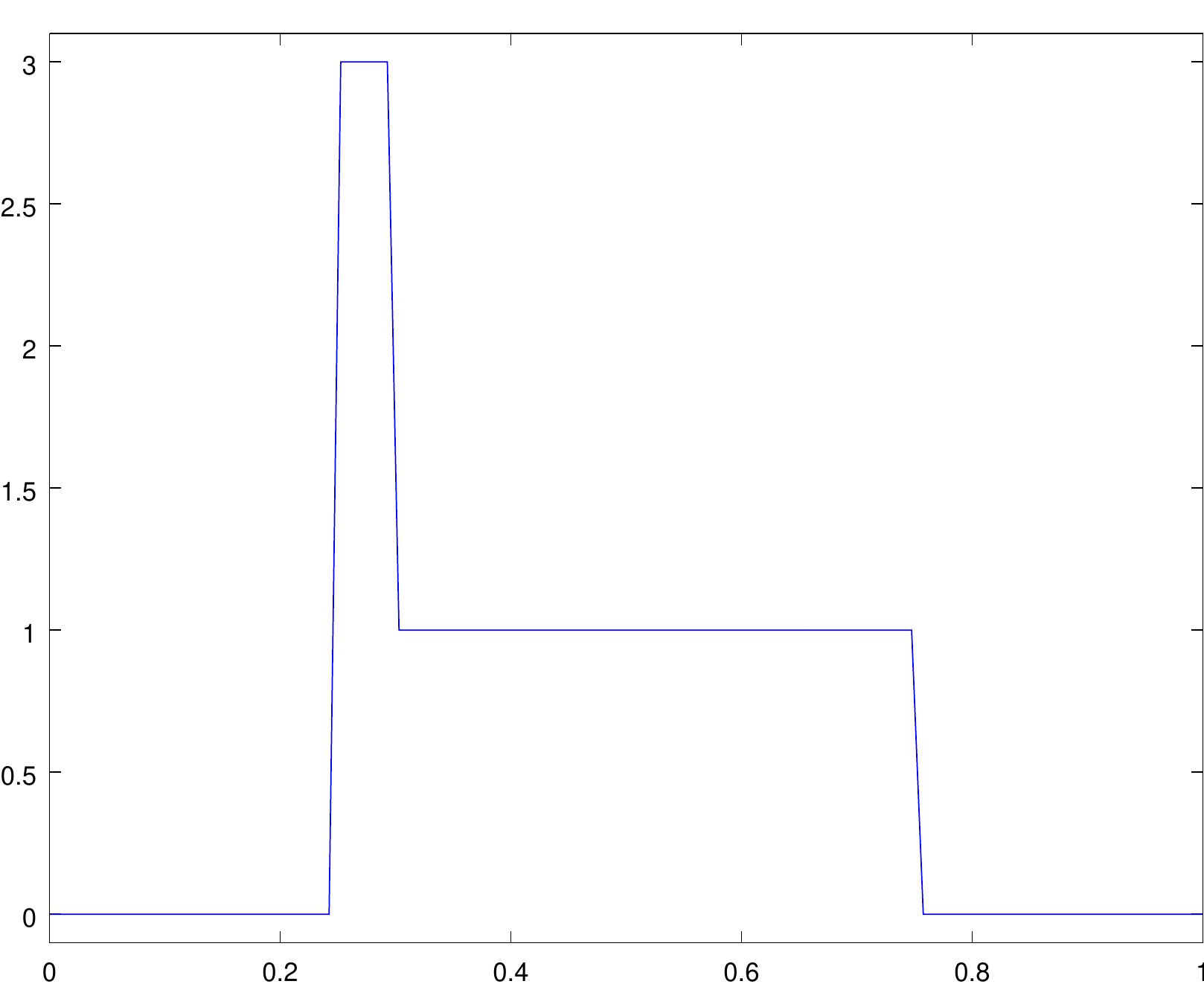}}
\:
\subfloat[$f_3$]{\includegraphics[clip=true,trim=2.6cm 7.5cm 2.55cm 7cm, width=0.315\textwidth]{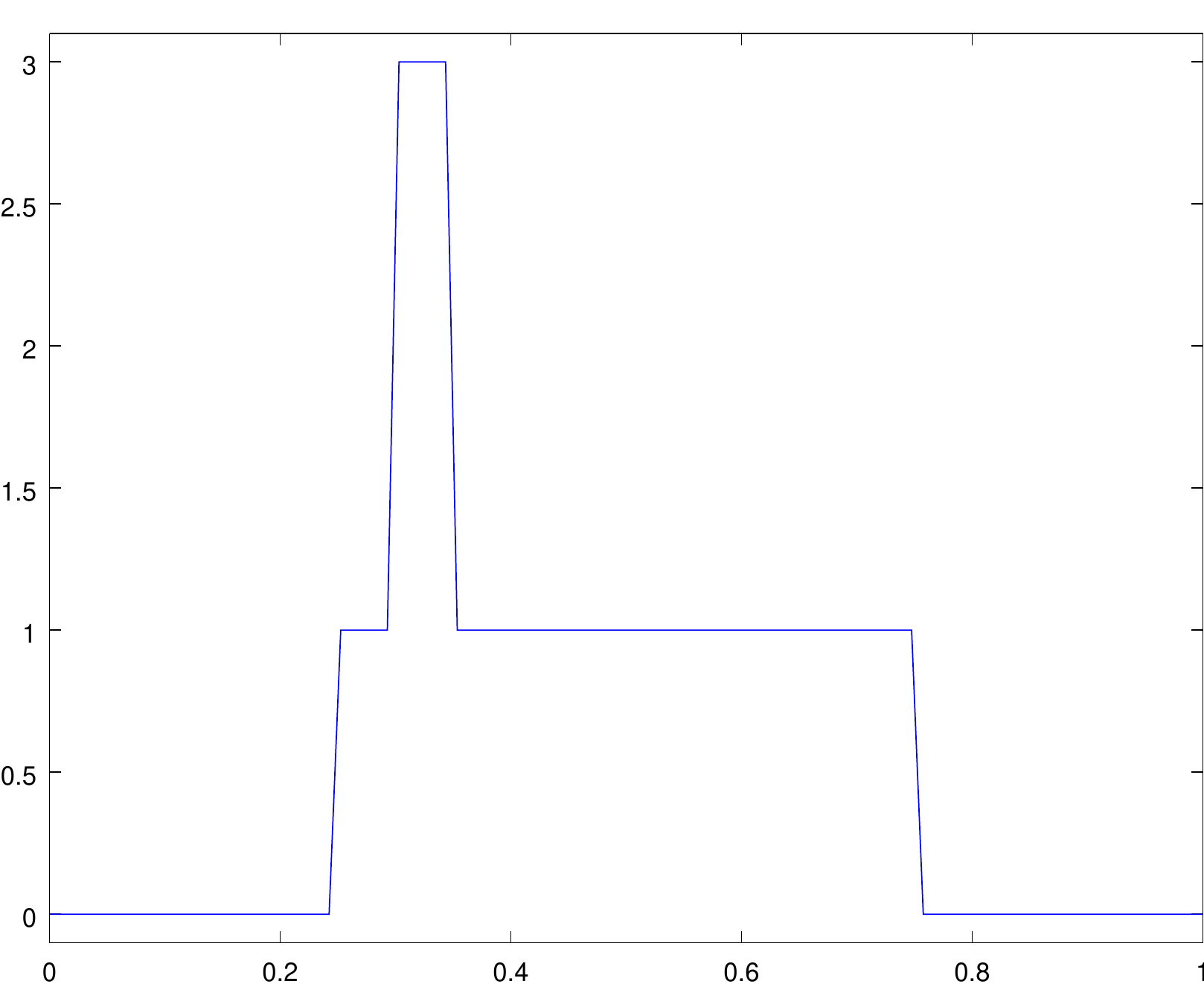}}
\caption[Three functions which can not be all registered with each other.]{Representation of the three functions $f_1$, $f_2$ and $f_3$ with $h=0.05$. the functions $f_2$ and $f_3$ are registered with respect to $f_1$. However $f_2$ and $f_3$ are not registered with each other, since it is more profitable to shift $f_2$ in order to align the highest parts of $f_2$ and $f_3$.} 
\label{fig:f1f2f3}
\end{figure}

When the congruent section exists, then the quotient can be included in a part $\Sc$ of the ambient space $M$ and the metric $d_M$ and $d_Q$ are corresponding.
The existence of a congruent section indicates us that the quotient space is not so complicated.
Indeed when there is an existence of a congruent section, the quotient space is embedded in the ambient space with respect to the distances in the quotient space and in the ambient space. In that case computations are easier, projecting data on this part $S$ and taking the mean. Then when such a congruent section does not exist, computing the Fréchet mean in quotient space is not so obvious. However, we can established proofs of inconsistency which are less tight. In~this article we prove that the method is inconsistent when the noise is large. 

\subsection{Inconsistency and Quantification of the Consistency Bias}
\label{sec:inc}
We start with \Cref{theo:Ks} which gives us an asymptotic behavior of the consistency bias when the noise level $\sigma$ tends to infinity. One key notion in \Cref{theo:Ks} is the concept of fixed point under the action $G$: a point $x\in M$ is a fixed point if for all $g\in G, \: g\cdot x=x$. We require that the support of the noise $\epsilon$ is not included in the set of fixed points. However, this condition is almost always fulfilled. For~instance in $\R^n$ the set of fixed points under a linear group action is a null set for the Lebesgue measure (unless the action is trivial: $g\cdot x=x$ for all $g\in G$ but this situation is irrelevant).  

\begin{Theorem}
\label{theo:Ks}
Let us suppose that the support of the noise $\epsilon$ is not included in the set of fixed points under the group action. Let $Y$ be the observable variable defined in Equation~(\ref{modgenforward}). If the Fréchet mean of $[Y]$ exists, then we have the following lower and upper bounds of the consistency bias noted $CB$:
\begin{equation}
\sigma K-2\|t_0\|\leq CB\leq \sigma K+2\|t_0\|,
\label{ineg}
\end{equation}
where $K=\underset{\|v\|=1}{\sup} \:\E\left( \underset{g\in G} \sup \:\psh{ v}{g\cdot\epsilon}\right)\in (0,1]$, $K$ is a constant which depends only of the standardized noise and of the group action.
The consistency bias has the following asymptotic behavior when the noise level $\sigma$ tends to infinity:
\begin{equation}
CB=\sigma K+o(\sigma) \mbox{ as } \sigma\to +\infini.
\label{equivalent}
\end{equation}
\end{Theorem}

In the following we note by $S$ the unit sphere of $M$. For $v\in S$, we call $\theta(v)=\E\left( \underset{g\in G} \sup\: \psh{v}{g\cdot \epsilon}\right)$, so that $K=\underset{v\in S}{\sup}\:\theta(v)$. The sketch of the proof is the following:
\begin{itemize} \item $K>0$ because the support of $\epsilon$ is not included in the set of fixed points under the action of $G$.
\item $K\leq 1$ is the consequence of the Cauchy-Schwarz inequality.
\item The proof of Inequalities~\eqref{ineg} is based on the triangular inequalities:
\begin{equation*}
\|m_\star\|-\|t_0\|\leq CB=\inf_{g\in G}\| t_0-g\cdot m_\star\|\leq \|t_0\|+\|m_\star\|,
\end{equation*}
where $m_\star$ minimizes $F$: having a piece of information about the norm of $m_\star$ is enough to deduce a piece of information about the consistency bias. 
\item The asymptotic Taylor expansion of the consistency bias~\eqref{equivalent} is the direct consequence of inequalities~\eqref{ineg}.
\end{itemize}

\begin{proof}[Proof of \Cref{theo:Ks}]
We note $S$ the unit sphere in $M$. In order to prove that $K>0$, we take $x$ in the support of $\epsilon$ such that $x$ is not a fixed point under the action of $G$. It exists $g_0\in G$ such that $g_0\cdot x\neq x$. We note $v_0=\frac{g_0\cdot x}{\|x\|}\in S $, we have $\psh{v_0}{g_0\cdot x}=\|x\|>\psh{v_0}{x}$ and by continuity of the dot product it exists $r>0$ such that: $
\forall y\in B(x,r)\quad \psh{v_0}{g_0\cdot y}>\psh{v_0}{y}
$ as $x$ is in the support of $\epsilon$ we have $\P(\epsilon \in B(x,r))>0$, it follows:
\begin{equation}
\P\left( \underset{g\in G}{\sup} \psh{v_0}{g\cdot \epsilon}>\psh{v_0}{\epsilon}\right)>0.
\label{inegv0}
\end{equation}

Thanks to Inequality~\eqref{inegv0} and the fact that $\sup_{g\in G} \psh{v_0}{g\cdot \epsilon}\geq \psh{v_0}{\epsilon}$ we have:
\begin{equation*}
\theta(v_0)= \E\left(\underset{g\in G}{\sup} \psh{v_0}{g \cdot \epsilon}\right)>\E(\psh{v_0}{\epsilon})=\psh{v_0}{\E(\epsilon)}=\psh{v_0}{0}=0.
\end{equation*}

Then we get $K\geq \theta(v_0)>0$. Moreover, if we use the Cauchy-Schwarz inequality:
\begin{equation*}
K\leq \sup_{v\in S} \E(\|v\|\times \|\epsilon\|)\leq \E(\| \epsilon\|^2)^{\frac12}=1.
\end{equation*}

In order to prove Inequalities~\eqref{ineg}, we use the "polar" coordinates of a point in $M$ (see \cref{fig:polar}), every point in $M$ can be represented by $(r,v)$ where $r\geq 0$ is the radius, and $v$ belong to $S$ the unit sphere in $M$, $v$ represents the ``angle''. We compute $F(m)$ as a function of $(r,v)$. In a first step, we minimize this expression as a function of $r$, in a second step we minimize this expression as a function of $v$. This~makes appear the constant $K$.

\begin{figure}[h]
\centering
\begin{tikzpicture}[scale=0.8]
\clip (-2.9,-2.7) rectangle (3.1,2.9);

\draw (-2.9,-2.7) rectangle (3.1,2.9);
\draw (-2.7,2.7) node{$M$};
\draw (-2.9,2.9) circle (0.6) ;

\draw (1.6,1.2) node{$\bullet$} node[above left] {$t_0$};
\draw (0,0)--(2.8,2.1);

\draw (0,0) node{$\bullet$} node[left] {$0$};

\draw[red] (2.1,0) node{$\bullet$} node[below left] {$\tilde \lambda(v) v$};
\draw[->](0,0)--(0.8,0);
\draw[thick] (0.8,0) node[above] {$v$};
\draw (0,0)--(3,0);

\draw[red] plot file {M.txt};

\draw[red] (2,2) node{$\bullet$} node[above right] {$m_\star$};
\draw[red] (2,-2) node{$\bullet$} node[below right] {$m_\star'$};
\draw[red] (-2,2) node{$\bullet$} node[below right] {$m_\star''$};
\draw[red] (-2,-2) node{$\bullet$} node[below left] {$m_\star'''$};
\draw[dashed] (1.6,1.2)--(2,2);
\end{tikzpicture}
\caption[Minimization of the variance on each half-line.]{We minimize the variance on each half-line $\R^+v$ where $\|v\|=1$. The element which minimizes the variance on such a half-line is $\tilde \lambda(v) v$, where $\tilde \lambda(v)\geq 0$. We get a surface in $M$ by $S\in v\mapsto \tilde \lambda(v) v$ (which is a curve in this figure since we draw it in dimension 2). The Proof of \Cref{theo:Ks} states that if $[m_\star]$ is a Fréchet mean then $m_\star$ is an extreme point of this surface. On this picture there are four extreme points which are in the same orbit: we took here the simple example of the group of rotations of $0$, $90$, $180$ and $270$ degrees.}
\label{fig:polar} 
\end{figure}
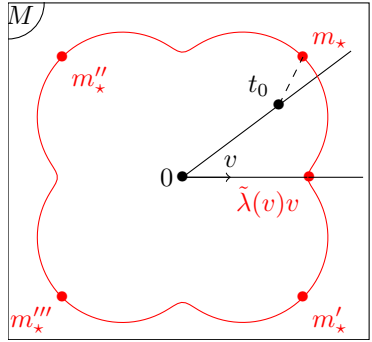

As we said, let us take $r\geq 0$ and $v\in S$, we expand the variance at the point $rv$:

\begin{equation}
F(r v)=\E\left( \underset{g\in G}{\inf} \|r v-g \cdot Y\|^2\right)=r^2-2r \E\left( \underset{g\in G}{\sup} \psh{v}{g \cdot Y}\right) +\E(\|Y\|^2).
 \label{expand}
\end{equation}

Indeed $\|g\cdot Y\|=\|Y\|$ thanks to the isometric action. We note $x^+=\max(x,0)$ the positive part of $x$. Moreover we define the two following functions:
\begin{equation*}
\lambda(v)=\E(\sup_{g\in G}\psh{v}{g \cdot Y})=\E(\sup_{g\in G}\psh{g \cdot Y}{v}) \mbox{ and } \tilde \lambda(v)=\lambda(v)^+ \mbox{ for } v\in S,
\end{equation*}
since that $f:x\in \R^+\mapsto x^2-2bx+c$ reaches its minimum at the point $r=b^+$ and $f(b^+)=c-(b^+)^2$, the $r_\star\geq0$ which minimizes~\eqref{expand} is $\tilde \lambda(v)$
and the minimum value of the variance restricted to the half line $\R^+v$ is:
\begin{equation*}
F(\tilde \lambda(v)v)=\E(\|Y\|^2)- \tilde \lambda(v)^2.
\end{equation*}

To find $[m_\star]$ the Fréchet mean of~$[Y]$, we need to maximize $\tilde \lambda(v)^2$ with respect to $v\in S$: 
\begin{equation*}
m_\star=\lambda(v_\star)v_\star \mbox{ with } v_\star\in \underset{v\in S}{\aM} \lambda(v).
\end{equation*}

Note that we remove the positive part and the square because $\aM \: \lambda=\aM \: (\lambda^+)^2$ indeed $\lambda$ takes a non negative value. In order to prove it let us remark that:
\begin{equation*}
\lambda(v)\geq \E(\psh{v}{\Phi\cdot t_0+\epsilon})=\psh{v}{\E(\Phi\cdot t_0)}+0,
\end{equation*}
then there is two cases: if $\E(\Phi \cdot t_0)=0$ then for any $v\in S$ we have $\lambda(v)\geq 0$, if $w=\E(\Phi \cdot t_0)\neq 0$ then we take $v=\frac{w}{\|w\|}\in S$, and we get $\lambda(v)\geq \psh{\frac{w}{\|w\|}}w=\|w\|\geq 0$.

As we said in the sketch of the proof we are interested in getting information about the norm of $\|m_\star\|$:
\begin{equation*}
\|m_\star\|=\lambda(v_\star)=\sup_{v\in S} \lambda.
\end{equation*}

Let $v\in S$, we have: $-\|t_0\|\leq \psh{v}{g\Phi\cdot t_0}\leq  \|t_0\|$ because the action is isometric. Now we decompose $Y=\Phi\cdot t_0+\sigma \epsilon$ and we get:
\begin{align}
\lambda(v)&=\E\left(\underset{g\in G}{\sup} \psh{v}{g\cdot Y}\right)=
\E\left(\underset{g\in G}{\sup} \left(\psh{v}{g\cdot \sigma\epsilon}+\psh{v}{g\Phi \cdot t_0 }\right)\right)\label{sumsup} \\
\lambda(v)&\leq \E\left(\underset{g\in G}{\sup} \left(\psh{v}{g \cdot \sigma\epsilon}+\|t_0\|\right) \right)  = \sigma\E\left(\underset{g\in G}{\sup} \psh{v}{g \cdot \epsilon}\right) +\|t_0\|  \label{cs1}\\
\lambda(v)&\geq \E\left(\underset{g\in G}{\sup} \left(\psh{v}{g\cdot \sigma\epsilon}\right) -\|t_0\|\right)  = \sigma \E\left(\underset{g\in G}{\sup} \psh{v}{g \cdot \epsilon}\right)-\|t_0\|. \label{cs2}
\end{align}

By taking the largest value in these inequalities with respect to $v\in S$, we get by definition of $K$:
\begin{equation}
-\|t_0\| +\sigma K \leq \|m_\star\|=\underset{v\in S}{\sup} \lambda(v) \leq \|t_0\|+\sigma K. \label{majthree}
\end{equation}

Moreover we recall the triangular inequalities:
\begin{equation}
\|m_\star\|-\|t_0\|\leq CB=\inf_{g\in G}\| t_0-g\cdot m_\star\|\leq \|t_0\|+\|m_\star\|,
\label{inegtri}
\end{equation}

Thanks to~\eqref{majthree} and to~\eqref{inegtri}, Inequalities~\eqref{ineg} are proved.
\end{proof}

\subsection{Remarks about \Cref{theo:Ks} and Its Proof}
We can ensure the presence of inconsistency as soon as the signal to noise ratio satisfies $\frac{\|t_0\|}{\sigma}<\frac{K}{2}$. Moreover, if the signal to noise ratio verifies $\frac{\|t_0\|}{\sigma}<\frac{K}{3}$ then the consistency bias is not smaller than $\|t_0\|$ i.e.,: $CB\geq  \|t_0\|$. In other words, the Fréchet mean in quotient space is too far from the template: the~template estimation with the Fréchet mean in quotient space is useless in this case.
In~\cite{dev2} we also gave lower and upper bounds as a function of $\sigma$ but these bounds were less informative than bounds given by \Cref{theo:Ks}. These bounds did not give the asymptotic behaviour of the consistency bias. Moreover, in~\cite{dev2} the lower bound goes to zero when the template becomes closed to fixed points. This~may suggest that the consistency bias was small for this kind of template. We prove here that it is not the case.

Note that \Cref{theo:Ks} is not a contradiction with~\cite{kur} where the authors proved the consistency of template estimation with the Fréchet mean in quotient space for all $\sigma>0$. Indeed their noise was included in the set of constant functions which are the fixed points under their group action. 

The constant $K$ appearing in the asymptotic behaviour of the consistency bias~\eqref{equivalent} is a constant of interest. We can give several (but similar) interpretations of $K$: 
\begin{itemize}
\item It follows from Equation~\eqref{ineg} that $K$ is the consistency bias with a null template $t_0=0$ and a standardized noise ($\sigma=1$).
\item From the proof of \Cref{theo:Ks} we know that $0<K\leq \E(\|\epsilon\|)\leq 1$. On the one hand, if $G$ is the group of rotations then $K=\E(\|\epsilon\|)$, 
because for all $v$ s.t. $\|v\|=1$, $\sup_{g\in G} \psh{v}{g\epsilon}=\|\epsilon\|$, by aligning $v$ and $\epsilon$. On the other hand if $G$ acts trivially (which means that $g\cdot x=x$ for all $g\in G,\: x\in M$) then $K=0$. The general case for $K$ is between two extreme cases: the group where the orbits are minimal (one point) and the group for which the orbits are maximal \mbox{(the whole sphere)}. We can state that the more the group action has the ability to align the elements, the larger the constant $K$ is  and the larger the consistency bias is.
\item The squared  quotient distance between two points is:
\begin{equation*}d_Q([a],[b])^2=\|a\|^2-2\sup_{g\in G} \psh{a}{g\cdot b}+\|b\|^2,
\end{equation*}
thus the larger $\sup_{g\in G} \psh{a}{g\cdot b}$, the smaller $d_Q([a],[b])$. $K=1-\frac12 \underset{\|v\|=1}{\inf} \E(d_Q^2([v],[ \epsilon]))$, encodes~the level of contraction of the quotient distance (or folding). The larger $K$ is, the more contracted the quotient space is.
\end{itemize}

One disadvantage of \Cref{theo:Ks} is that it ensures the presence of inconsistency for $\sigma$ large enough but it says nothing when $\sigma$ is small, in this case one can refer to~\cite{mio2} or~\cite{dev2}.

We can remark that this Theorem can be used as an alternating proof the following Theorem (which was already proved in ~\cite{dev2}), proving and quantifying inconsistency when the template is a fixed point:
\begin{Corollary}
\label{prop:fixedpoint}
Let $G$ acting isometrically on $M$ an Hilbert space. Let $t_0$ be a fixed point, and $\epsilon$ a standardized noise which support is not included in the set of fixed points. Then estimating the template with the Fréchet mean is inconsistent. Moreover if the Fréchet mean in quotient space exists then the consistency bias is equal to:
\begin{equation*}
CB=\sigma K.
\end{equation*}
\end{Corollary}
Indeed for $t_0=0$ which is a particular fixed point we have $CB=\sigma K$ thanks to \Cref{theo:Ks}. If $t_0$ is a fixed point non necessarily equal to $0$, we can define $Y'=Y-t_0=0+\sigma \epsilon$, in this random variable $0$ is the template we can apply the formula $CB=\sigma K$ to the random variable $Y'$, which concludes.

In the proof of \Cref{theo:Ks}, we have seen that the minimum of the variance restricted to the half-line $\R^+v$ for $v\in S$, was \begin{equation*}
\E(\|Y\|^2)-\left(\left(\E\left(\underset{g\in G}{\inf} \psh{v}{g\cdot Y}\right)\right)^+\right)^2.
\end{equation*}
therefore $\tilde \lambda(v)=\left(\E\left(\underset{g\in G}{\inf} \psh{v}{g\cdot Y}\right)\right)^+$ is a registration score: $\tilde \lambda(v)$ 
tells you how much it is a good idea to search the Fréchet mean of $[Y]$ in the direction pointed by $v$: the more $\tilde \lambda(v)$ is large, the more $v$ is a good choice. On the contrary when this value is equal to zero, it is useless to search the Fréchet mean in this direction.

Likewise, for $v\in S$, $\theta(v)=\E(\underset{g\in G}{\sup} \psh{g\cdot v}{\epsilon})$ is a registration score with respect to the noise, the~larger $\theta(v)$, the more the unit vector $v$ looks like to the noise $\epsilon$ after registration.

If $[m_\star]$ is a Fréchet mean of $[Y]$ we have seen that its norm verifies:
\begin{equation*}
\|m_\star\|=\underset{\|v\|=1}{\sup} \E(\sup_{g\in G} \psh{v}{g\cdot Y}).
\end{equation*}

Then if there is two different Fréchet means of $[Y]$ noted $[m_\star]$ and $[n_\star]$, we can deduce that $\|m_\star\|=\|n_\star\|$. Even if there is no uniqueness of the Fréchet mean in the quotient space, we can state that the representants of the different Fréchet means have all the same norm.
\begin{Remark}
We can also wonder if the converse of \Cref{theo:Ks} is true: if $\epsilon$ is a non biased noise always included in the set of fixed point, is $[t_0]$ a Fréchet mean of $[\Phi\cdot t_0+\sigma\epsilon]$?  A simple computation show that $t_0$ is a minimum of the variance:
\begin{align}
F(m)&=\E\left(\inf_{g\in G} \|m-g\cdot (\Phi t_0+\sigma\epsilon)\|^2\right)\nonumber\\
&=\|m\|^2+\E(\|\Phi t_0+\sigma\epsilon\|^2)-2\E (\sup_g \psh{m}{g\Phi t_0} +\psh{m}{g\sigma\epsilon})\nonumber\\
&=\|m\|^2+\E(\|\Phi t_0+\sigma\epsilon\|^2)-2\E \left(\sup_{g\in G} \psh{m}{g\cdot t_0}\right) -2\psh{m}{\E(\sigma\epsilon)} \nonumber\\
&=\|m\|^2+\E(\|\Phi t_0+\sigma\epsilon\|^2)-2\E \left(\sup_{g\in G} \psh{m}{g\cdot t_0}\right)\label{minse}
\end{align}

We see that the element $m$ which minimizes~\eqref{minse} does not depend of $\sigma$, in particular we can assume $\sigma=0$, and wonder which elements minimizes $F(m)=\E(\inf_{g\in G} \|m-g\Phi \cdot t_0\|^2)$, it becomes clear that only the points in the orbit of $t_0$ can minimize this variance. Then when $\epsilon$ is included in the set of fixed points, the estimation is always consistent for all $\sigma$. This is an alternative proof of the Theorem of consistency done by Kurtek et al.~\cite{kur}.
\end{Remark}

In the proof of \Cref{theo:Ks}, we have seen that the direction of the Fréchet mean of $[Y]$ is given by the supremum of this quantity~\eqref{sumsup}:
\begin{equation*}
\E\left( \sup_{g\in G} \psh{v}{g\cdot \sigma \epsilon} +\psh{v}{g\Phi \cdot t_0}\right).
\end{equation*}

This Equation is a good illustration of the difficulty to compute the Fréchet mean in quotient space. Indeed, we have on one side the contribution of the noise $\psh{v}{g\cdot \sigma\epsilon}$ and on the other side the contribution of the template $\psh{v}{g\Phi\cdot t_0}$, and we take the supremum of the sum of these two contributions over $g\in G$. Unfortunately the supremum of the sum of two terms is not equal to the sum of the supremum of each of these terms. Hence, it is difficult to separate these two contributions. However, we can intuit that when the noise is large, $\psh{v}{g\sigma \epsilon}$ prevails over $\psh{v}{g\Phi \cdot t_0}$, and the use of the Cauchy-Schwarz inequality in Equations~\eqref{cs1} and~\eqref{cs2} proves it rigorously. We can conclude that, when the noise is large, the direction of the Fréchet mean in the quotient space depends more on the noise than on the template.

\subsection{Template Estimation with the Max-Max Algorithm}
\vspace{-6pt}
\label{sec:max}
\subsubsection{Max-Max Algorithm Converges to a Local Minima of the Empirical Variance}
\Cref{sec:inc} can be understood as follows: if we want to estimate the template by minimizing the Fréchet mean with quotient space, then there is a bias. This supposes that we are able to compute such a Fréchet mean. In practice, we cannot minimize the exact variance in quotient space, because we have only a finite sample and not the whole distribution. In this section we study the estimation of the empirical Fréchet mean with the max-max algorithm. We assume that the group is finite. In this case, the registration can always be found by an exhaustive search. Hence, the numeric experiments which we conduct in \Cref{sec:sim} lead to an empirical Karcher mean in a finite number of steps.
In~a compact group acting continuously, the registration also exists but is not necessarily computable without approximation. 

If we have a sample: $Y_1,\tp, Y_I$ of independent and identically distributed copies of $Y$, then we define the empirical variance in the quotient space:
\begin{equation}
  M\ni x\mapsto   F_I(x)=\frac1I \sum_{i=1}^I d^2_Q([x],[Y_i])=\frac{1}{I} \sum_{i=1}^I \underset{g_i\in G}{\min} \|x-g_i\cdot Y_i\|^2=\frac{1}{I} \sum_{i=1}^I \underset{g_i\in G}{\min} \|g_i\cdot x-Y_i\|^2.
    \label{empvar}
\end{equation}

The empirical variance is an approximation of the variance. Indeed thanks to the law of large number we have $\lim_{I\to \infini} F_I(x)=F(x)$ for all $x\in M$. One element which minimizes globally (respectively locally) $F_I$ is called an empirical Fréchet mean (respectively an empirical Karcher mean). For $x\in M$ and $\underline g\in G^I$: $\underline g=(g_1,\tp, g_I)$ where $g_i\in G$ for all $i=1..I$ we define $J$ an auxiliary function by:
\begin{equation*}
 J(x, \underline g)=\frac1I \underset{i=1}{\overset{I}{\sum}} \|x-g_i\cdot Y_i\|^2=\frac1I \underset{i=1}{\overset{I}{\sum}} \|g_i^{-1}\cdot x- Y_i\|^2.
\end{equation*}

The max-max~\cref{algo} iteratively minimizes the function $J$ in the variable $x\in M$ and in the variable $\underline g\in G^I$ (see \cref{fig:J}):

\begin{algorithm}[h]
\caption{Max-Max Algorithm} 
\label{algo}
\begin{algorithmic}
\REQUIRE A starting point $m_0\in M$, a sample $Y_1,\tp,Y_I$.
\vspace{1mm}
\STATE $n=0$.
\vspace{1mm}
\WHILE{Convergence is not reached}
\vspace{1mm}
\STATE Minimizing $\underline g\in G^I\mapsto J(m_n,\underline g)$: we get $g_i^{n}$ by registering $Y_i$ with respect to $m_n$.
\vspace{1mm}
\STATE Minimizing $x\in M\mapsto J(x,\underline g^n)$: we get $m_{n+1}=\frac1I \underset{i=1}{\overset{I}{\sum}} g_i^n Y_i$.
\vspace{1mm}
\STATE $n=n+1$.
\vspace{1mm}
\ENDWHILE
\vspace{1mm}
\STATE $\hat m=m_{n}$
\end{algorithmic}
\end{algorithm}

First, we note that this algorithm is sensitive to the the starting point. However we remark that $m_1=\frac{1}{I}\sum_{i=1}^I g_i\cdot Y_i$ for some $g_i\in G$, thus without loss of generality, we can start from $m_1=\frac{1}{I} \sum_{i=1}^I g_i\cdot Y_i$ for some $g_i\in G$.
The empirical variance does not increase at each step of the algorithm since: 
\begin{equation*}
 F_I(m_n)=J(m_n,\underline g^n)\geq J(m_{n+1},\underline g^n)\geq J(m_{n+1},\underline g^{n+1})=F_I(m_{n+1})
\end{equation*}

\begin{Proposition}
\label{prop:fns}
As the group is finite, the convergence is reached in a finite number of steps.
\end{Proposition}

\begin{proof}[Proof of \Cref{prop:fns}]
The sequence $(F_I(m_n))_{n\in \N}$ is non-increasing. Moreover the sequence $(m_n)_{n\in \N}$ takes value in a finite set which is:
$
\{\frac1I \sum_{i=1}^I g_i\cdot Y_i,\: g_i\in G\}.
$
Therefore, the sequence $(F_I(m_n))_{n\in\N}$ is stationary. Let $n\in \N$ such that $F_I(m_n)=F_I(m_{n+1})$. 
Hence the empirical variance did not decrease between step $n$ and step $n+1$ and we have:
\begin{equation*}
F_I(m_n)=J(m_n,\underline g_n)=J(m_{n+1},\underline g_n)=J(m_{n+1},\underline{g}_{n+1})=F_I(m_{n+1}),
\end{equation*}
as $m_{n+1}$ is the unique element which minimizes $m\mapsto J(m,\underline{g}_{n})$ we conclude that $m_{n+1}=m_n$. 
\end{proof}

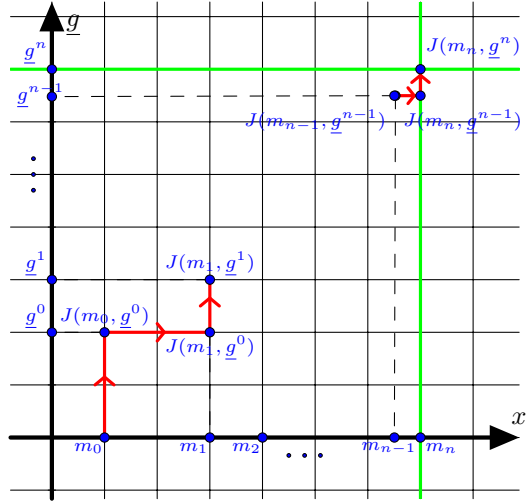
\begin{figure}[h]
    \centering
   \begin{tikzpicture}[line cap=round,line join=round,>=triangle 45,x=1cm,y=1cm,scale=0.7]
\draw [color=black, xstep=1.0cm,ystep=1.0cm] (-0.7852537722908061,-1.1625514403292634) grid (8.874746227709274,8.277448559670814);
\draw[->,ultra thick,color=black] (-0.7852537722908061,0.) -- (8.874746227709274,0.);
\foreach \x in {,1.,2.,3.,4.,5.,6.,7.,8.}
\draw[shift={(\x,0)},color=black] (0pt,-2pt);
\draw[color=black] (8.55474622770927,0.08) node [anchor=south west] {$x$};
\draw[->,ultra thick,color=black] (0.,-1.1625514403292634) -- (0.,8.277448559670814);
\foreach \y in {-1.,1.,2.,3.,4.,5.,6.,7.,8.}
\draw[shift={(0,\y)},color=black] (2pt,0pt) -- (-2pt,0pt);
\draw[color=black] (0.1,7.83744855967081) node [anchor=west] {$\underline g$};
\clip(-0.7852537722908061,-1.1625514403292634) rectangle (8.874746227709274,8.277448559670814);
\draw [line width=1.2pt,color=red] (1.,0.)-- (1.,2.);
\draw [line width=1.2pt,color=red] (1.,1.15) -- (1.18,1.);
\draw [line width=1.2pt,color=red] (1.,1.15) -- (0.82,1.);
\draw [dash pattern=on 5pt off 5pt] (1.,2.)-- (0.,2.);
\draw [dash pattern=on 5pt off 5pt] (3.,2.)-- (3.,0.);
\draw [line width=1.2pt,color=red] (1.,2.)-- (3.,2.);
\draw [line width=1.2pt,color=red] (2.135,2.) -- (2.,1.835);
\draw [line width=1.2pt,color=red] (2.135,2.) -- (2.,2.165);
\draw [dash pattern=on 5pt off 5pt] (3.,3.)-- (0.,3.);
\draw [dash pattern=on 5pt off 5pt] (6.524746227709255,6.497448559670791)-- (0.,6.48);
\draw [line width=1.2pt,color=green,domain=-0.7852537722908061:8.874746227709274] plot(\x,{(--49.-0.*\x)/7.});
\draw [line width=1.2pt,color=green] (7.,-1.1625514403292634) -- (7.,8.277448559670814);
\draw [line width=2.pt,color=red] (6.504746227709255,6.497395075415429)-- (6.524746227709255,6.497448559670791);
\draw [dash pattern=on 5pt off 5pt] (6.524746227709255,6.497448559670791)-- (6.504746227709255,0.);
\draw [line width=1.2pt,color=red] (3.,2.)-- (3.,3.);
\draw [line width=1.2pt,color=red] (3.,2.65) -- (3.18,2.5);
\draw [line width=1.2pt,color=red] (3.,2.65) -- (2.82,2.5);
\draw [line width=1.2pt,color=red] (6.524746227709255,6.497448559670791)-- (7.,6.497448559670791);
\draw [line width=1.2pt,color=red] (6.882373113854629,6.49744855967079) -- (6.762373113854628,6.347448559670789);
\draw [line width=1.2pt,color=red] (6.882373113854629,6.49744855967079) -- (6.762373113854628,6.647448559670791);
\draw [line width=1.2pt,color=red] (7.,6.497448559670791)-- (7.,7.);
\draw [line width=1.2pt,color=red] (7.,6.883724279835396) -- (7.165,6.748724279835395);
\draw [line width=1.2pt,color=red] (7.,6.883724279835396) -- (6.835,6.748724279835395);
\begin{scriptsize}
\draw [fill=blue] (1.,0.) circle (2.5pt);

\draw [fill=blue] (1.,2.) circle (2.5pt);
\draw [fill=blue] (0.,2.) circle (2.5pt);
\draw [fill=blue] (3.,2.) circle (2.5pt);
\draw [fill=blue] (3.,0.) circle (2.5pt);
\draw [fill=blue] (3.,3.) circle (2.5pt);
\draw [fill=blue] (0.,3.) circle (2.5pt);
\draw [fill=blue] (6.524746227709255,6.497448559670791) circle (2.5pt);
\draw [fill=blue] (7.,0.) circle (2.5pt);
\draw [fill=blue] (0.,7.) circle (2.5pt);
\draw [fill=blue] (7.,7.) circle (2.5pt);
\draw [fill=blue] (4.,0.) circle (2.5pt);
\draw [fill=blue] (6.504746227709255,0.) circle (2.5pt);
\draw [fill=blue] (0.,6.48) circle (2.5pt);
\draw [fill=blue] (4.5,-0.34) circle (1.0pt);
\draw [fill=blue] (4.8,-0.34) circle (1.0pt);
\draw [fill=blue] (5.1,-0.34) circle (1.0pt);
\draw [fill=blue] (6.504746227709255,6.497395075415429) circle (2.5pt);
\draw [fill=blue] (-0.35,5.3) circle (1.0pt);
\draw [fill=blue] (-0.35,5.) circle (1.0pt);
\draw [fill=blue] (-0.35,4.7) circle (1.0pt);
\draw [fill=blue] (7.,6.497448559670791) circle (2.5pt);

\draw[color=blue] (0.7,-0.2) node {$m_0$};
\draw[color=blue] (2.7,-0.2) node {$m_1$};
\draw[color=blue] (3.7,-0.2) node {$m_2$};
\draw[color=blue] (6.4,-0.2) node {$m_{n-1}$};
\draw[color=blue] (7.4,-0.2) node {$m_n$};

\draw[color=blue] (-0.3,2.3) node {$\underline g^0$};
\draw[color=blue] (-0.3,3.3) node {$\underline g^1$};
\draw[color=blue] (-0.21,6.5) node {$\underline g^{n-1}$};
\draw[color=blue] (-0.3,7.3) node {$\underline g^n$};

\draw[color=blue] (1,2.3) node {$J(m_0,\underline g^0)$};
\draw[color=blue] (3,1.7) node {$J(m_1,\underline g^0)$};
\draw[color=blue] (3,3.3) node {$J(m_1,\underline g^1)$};

\draw[color=blue] (5,6.0) node {$J(m_{n-1},\underline{g}^{n-1})$};
\draw[color=blue] (7.8,6.0) node {$J(m_n,\underline g^{n-1})$};
\draw[color=blue] (8,7.4) node {$J(m_n,\underline g^n)$};

\end{scriptsize}
\end{tikzpicture}
   
    \caption[Iterative minimization of the function $J$ in order to estimate empirical Fréchet mean.]{Iterative minimization of the function $J$ on the two axis, the horizontal axis represents the variable in the space $M$, the vertical axis represents the set of all the possible registrations $G^I$. Once the convergence is reached, the point $(m_n,g_n)$ is the minimum of the function $J$ on the two axis in green. Is~this point the minimum of $J$ on its whole domain? There are two pitfalls: firstly this point could be a saddle point, it can be avoided with \Cref{prop:regu}, secondly this point could be a local (but not global) minimum, this is discussed in \Cref{subsec:local}}
    \label{fig:J}  
\end{figure}

This proposition gives us a shutoff parameter in the max-max algorithm: we stop the algorithm as soon as $m_n=m_{n+1}$.
Let us call $\hat m$ the final result of the max-max algorithm. It may seem logical that $\hat m$ is at least a local minimum of the empirical variance. However this intuition may be wrong: let us give a toy counterexample, suppose that we observe $Y_1,\tp,Y_I$, due to the transformation of the group it is possible that $\sum_{i=1}^n Y_i=0$. We can start from $m_1=0$ in the max-max algorithm, as $Y_i$ and $0$ are already registered, the max-max algorithm does not transform $Y_i$. At step two, we still have $m_2=0$, by~induction the max-max algorithm stays at $0$ even if $0$ is not a Fréchet or Karcher mean of $[Y]$. Because~$0$ is equally distant from all the points in the orbit of $Y_i$, $0$ is called a focal point of $[Y_i]$. The~notion of focal point is important for the consistency of the Fréchet mean in manifold~\cite{bha}. Fortunately, the situation where $\hat m$ is not a Karcher mean is almost always avoided due to the following~statement:
\begin{Proposition}
\label{prop:regu}
Let $\hat{m}$ be the result of the max-max algorithm. If the registration of $Y_i$ with respect to $\hat{m}$ is unique, in other words, if $\hat m$ is not a focal point of $Y_i$ for all $i\in 1..I$ then
$\hat{m}$ is a local minimum of $F_I$: $[\hat{m}]$ is an empirical Karcher mean of $[Y]$.
\end{Proposition}
Note that, if we call $z$ the registration of $y$ with respect to $m$, then the registration is unique if and only if $\psh{m}{z-g\cdot z}\neq 0$ for all $g\in G\setminus\{e\}$. Once the max-max algorithm has reached convergence, it~suffices to test this condition for $\hat{m}$ obtained by the max-max algorithm and $Y_i$ for all $i$. This condition is in fact generic and is always obtained in practice.
\begin{proof}[Proof of \Cref{prop:regu}]
We call $g_i$ the unique element in $G$ which register $Y_i$ with respect to $\hat m$, for all $h\in G\setminus\{g_i\}$, $\|\hat m-g_i\cdot Y_i\|<\|\hat m-h_i\cdot Y_i\|$. By continuity of the norm we have for $a$ close enough to $m$: $\|a-g_i\cdot Y_i\|<\|a-h_i\cdot Y_i\|$ for all $h_i \neq g_i$ (note that this argument requires a finite group). The~registrations of $Y_i$ with respect to $m$ and to $a$ are the same: 
\begin{equation*}
    F_I(a)=\frac1I \sum_{i=1}^I \|a-g_i\cdot Y_i\|^2=J(a,\underline g)\geq J(\hat{m},\underline g)=F_I(\hat{m}),
\end{equation*} 
because $m\mapsto J(m,\underline g)$ has one unique local minimum $\hat{m}$.
\end{proof}

\begin{Remark} We remark the max-max algorithm is in fact a gradient descent. The gradient descent is a general method to find the minimum of a differentiable function. Here we are interested in the minimum of the variance $F$: let $m_0\in M$ and we define by induction the gradient descent of the variance $m_{n+1}=m_n-\rho \nabla F(m_n)$, where~$\rho>0$ and $F$ the variance in the quotient space. In~\cite{dev2}, the gradient of the variance in quotient space for finite group and for a regular point $m$ was computed ($m$ is regular as soon as $g\cdot m=m$ implies $g=e$), this~leads to:
\begin{equation*}
    m_{n+1}=m_n-2\rho\left[m_n-\E(g(Y,m_n)\cdot Y)\right],
\end{equation*}
where $g(Y,m_n)$ is the almost-surely unique element of the group which registers $Y$ with respect to $m_n$. Now if we have a set of data $Y_1,\tp, Y_n$ we can approximated the expectation which leads to the following approximated gradient descent:
\begin{equation*}
    m_{n+1}= m_n(1-2\rho) +\rho \frac2I \sum_{i=1}^I g(Y_i,m_n)\cdot Y_i,
\end{equation*}
now by taking $\rho=\frac12$ we get $m_{n+1}= \frac1I \sum_{i=1}^I g(Y_i,m_n)\cdot Y_i$. So the approximated gradient descent with $\rho=\frac12$ is exactly the max-max algorithm. However, the max-max algorithm for finite group, is proved to be converging in a finite number of steps which is not the case for gradient descent in general.
\end{Remark}

\subsection{Simulation on Synthetic Data}
\label{sec:sim}

In this Section, we consider data in an Euclidean space $\R^N$ equipped with its canonical dot product $\psh{\cdot }{\cdot}$, and $G=\Z/\N\Z$ acts on $\R^N$ by time translation on coordinates:
\begin{equation*}
(\bar{k}\in \Z/N\Z, (x_1,\tp,x_N)\in \R^N)\mapsto (x_{1+k},x_{2+k},\tp x_{N+k}),
\end{equation*}
where indexes are taken modulo $N$. This space models the discretization of functions defined on $[0,1]$ with $N$ points. This action is found in~\cite{all} and used for neuroelectric signals in~\cite{hit}. The registration between two vectors can be made by an exhaustive research but it is faster with the fast Fourier transform~\cite{fft}.  

\subsubsection{Max-Max Algorithm with a Step Function as Template}

We display an example of a template and template estimation with the max-max algorithm on Figure \ref{F6}a.
This experiment was already conducted in~\cite{all}, but no explanation of the appearance of the bias was provided.
We know from \Cref{sec:max} that the max-max output is an empirical Karcher mean, and that this result can be obtained in a finite number of steps. Taking $\sigma=10$ may seem extremely high, however the standard deviation of the noise at each point is not $10$ but $\frac{\sigma}{\sqrt N}=1.25$ which is reasonable.

The sample size is $10^5$, the algorithm stopped after 247 steps, and $\hat m$ the estimated template (in red on the Figure \ref{F6}a) is not a focal points of the orbits $[Y_i]$, then \Cref{prop:regu} applies. We call empirical bias (noted EB) the quotient distance between the true template and the point $\hat m$ given by the max-max result. On this experiment we have $\frac{EB}{\sigma}\simeq 0.11$. Of course, one could think that we estimate the template with an empirical bias due to a too small sample size which induces fluctuation. To~reply to this objection, we keep in memory~$\hat{m}$ obtained with the max-max algorithm. If there was no inconsistency then we would have $F(t_0)\leq F(\hat{m})$. We do not know the value of the variance $F$ at these points, but thanks to the law of large number, we know that:
\begin{equation*}
F(t_0)=\underset{I\to \infini}{\lim} F_I(t_0) \mbox{ and } F(\hat{m})=\underset{I\to \infini}{\lim} F_I(\hat{m}),
\end{equation*}

Given a sample, we compute $F_I(t_0)$ and $F_I(\hat{m})$ thanks to the definition of the empirical variance $F_I$~\eqref{empvar}. We display the result on Figure \ref{F6}b, this tends to confirm that $F(t_0)>F(\hat{m})$. In other word, the~variance at the template is larger that the variance at the point given by the max-max algorithm. 

\begin{figure}[h]
\centering
\subfloat[]{\includegraphics[clip=true,trim=2.6cm 7cm 2.55cm 6cm,width=0.45\textwidth]{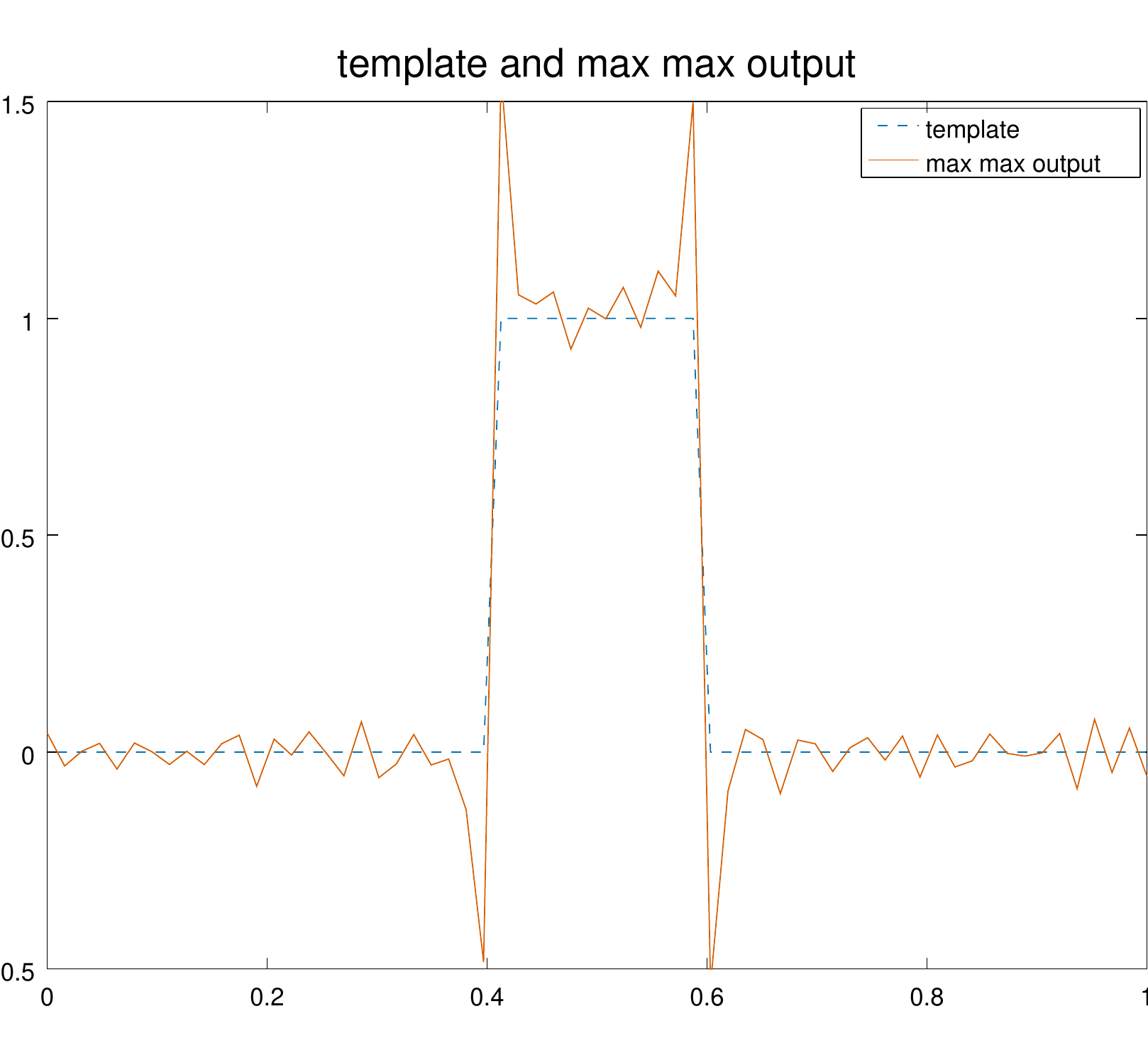} \label{fig:bias}}
\quad
\subfloat[]{\includegraphics[clip=true,trim=2.6cm 7cm 2.55cm 6cm,width=0.45\textwidth]{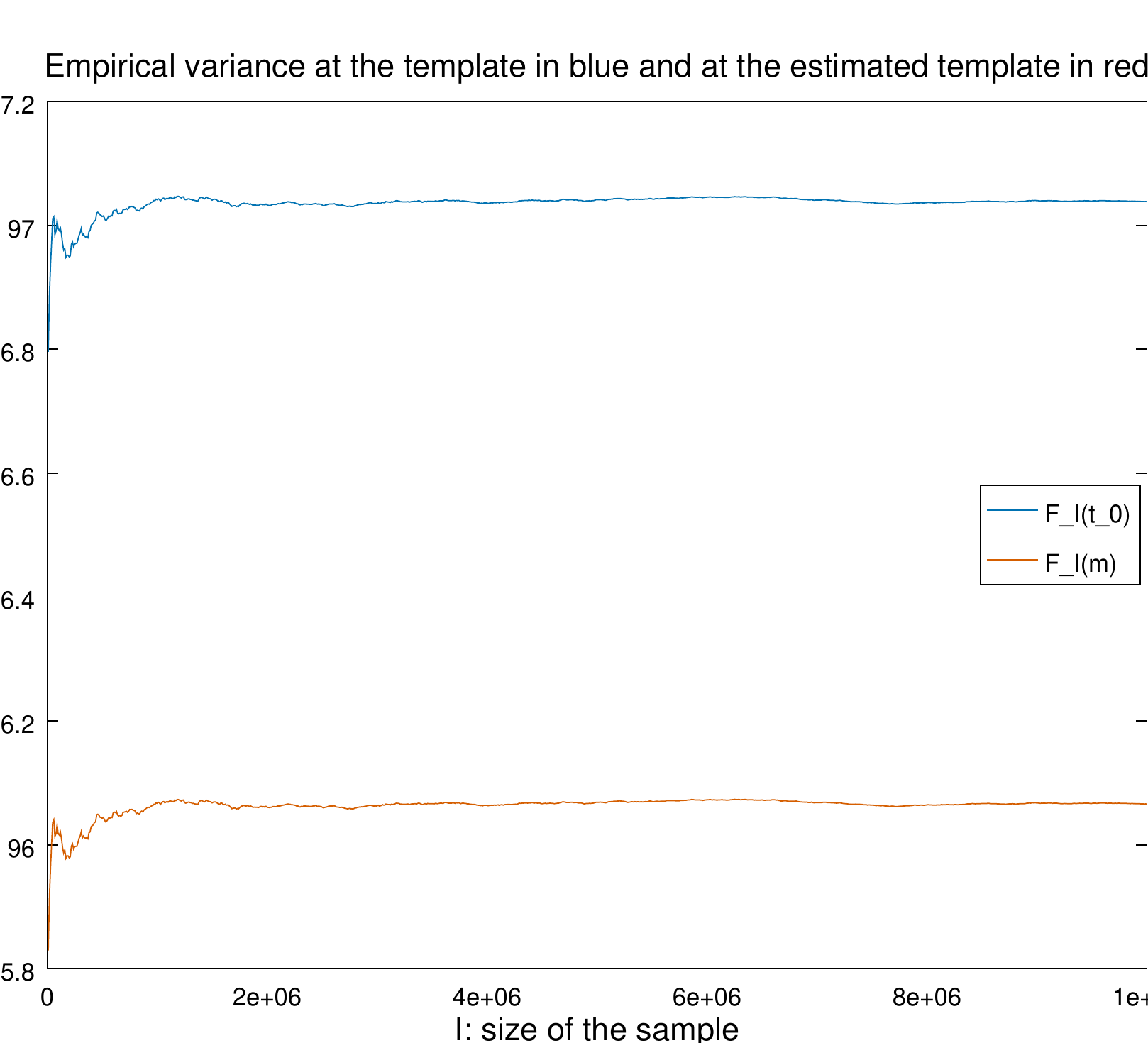}\label{fig:variance}} 
\caption[Template and estimated template and its empirical variances.]{Template $t_0$ and template estimation $\hat m$ on Figure \ref{F6}a. Empirical variance at the template and template estimation with the max-max algorithm as a function of the size of the sample on Figure \ref{F6}b. (\textbf{a}) Example of a template (a step function) and the estimated template $\hat m$ with a sample size $10^5$ in $\R^{64}$, $\epsilon$ is Gaussian noise and $\sigma=10$. At the discontinuity points of the template, we observe a Gibbs-like 
    phenomena; (\textbf{b}) Variation of $F_I(t_0)$ (in blue) and of $F_I(\hat m)$ (in red) as a function of $I$ the size of the sample. 
Since convergence is already reached, $F(\hat m)$, which is the limit of red curve, is below $F(t_0)$: $F(t_0)$ is the limit of the blue curve. Due to the inconsistency, $\hat m$ is an example of point such that~$F(\hat m)<F(t_0)$.}

\label{F6}
\end{figure}

\subsubsection{Max-Max Algorithm with a Continuous Template}

\Cref{F6}a shows that the main source of the inconsistency was the discontinuity of the template. One may think that a continuous template would lead to a better behaviour. However, it is not the case as presented on Figure \ref{fig:bias3}. Even with a large number of observations created from a continuous template we do not observe a convergence to the template. In the example of Figure \ref{fig:bias3}, the empirical bias satisfies $\frac{EB}{\sigma}=0.23$. In green we also display the mean of data knowing transformations, this produces a much better result, since that in this case we have $\frac{EB}{\sigma}=0.04$.
\begin{figure}[h]
    \centering
    \includegraphics[scale=0.4,clip=true,trim=2.6cm 7.5cm 2.5cm 6cm]{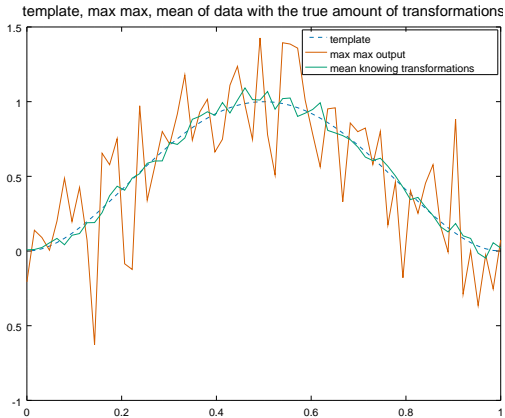}
    \caption[Continuous template and template estimation, and mean knowing transformations]{Example of an other template (here a discretization of a continuous function) and template estimation with a sample size $10^3$ in $\R^{64}$, $\epsilon$ is Gaussian noise and $\sigma=10$. Even with a continuous function the inconsistency appears.}
    \label{fig:bias3}
\end{figure}

\subsubsection{Does the Max-Max Algorithm Give Us a Global Minimum or Only a Local Minimum of the~Variance?}
\label{subsec:local}
\Cref{prop:regu} tells us that the output of the max-max algorithm is a Karcher mean of the variance, but we do not know whether it is Fréchet mean of the variance. In other words, is the output a global minimum of the variance? 
In fact, $F_I$ has a lot of local minima which are not global. To illustrate this, we may use the max-max algorithm with different starting points and we observe different outputs (which are all local minima thanks to \Cref{prop:regu}) with different empirical variance on \Cref{tab:variances}.

\begin{table}[h]
\centering
\begin{tabular}{ccccccc}

\textbf{Points} & \textbf{Template} \boldmath{$t_0$} & \boldmath{$\hat m_1$} & \boldmath{$\hat m_2$} & \boldmath{$\hat m_3$} & \boldmath{$\hat m_4$} &\boldmath{ $\hat m_5$}\\

 Empirical variance at these points& 96.714& 95.684&    95.681&   95.676&   95.677&   95.682\\
 
\end{tabular}
\caption{Empirical variances at $5$ different outputs of the max-max algorithm coming from the same sample of size $10^4$, but with different starting points. We use $\sigma=10$ and the action of time translation in $\R^{64}$. Conclusion: on these five points, only $\hat m_3$ is an eventual global minima.}
\label{tab:variances}
\end{table}

\section{Inconsistency in the Case of Non Invariant Distance under the Group Action}
\label{sec:noni}
\vspace{-6pt}
\subsection{Notation and Hypothesis}
In this Section, data still come from an Hilbert space $M$. However, we take a group of deformation $G$ which acts in a non invariant way on $M$. 
Starting from a template $t_0$ we consider a random deformation in the group $G$ namely a random variable $\Phi$ which takes value in $G$ and $\epsilon$ an standardized noise in $M$ independent of $\Phi$. We suppose that our observable random variable is:
\begin{equation*}
    Y=\Phi\cdot  t_0+\sigma\epsilon \mbox{ with } \sigma>0,\: \E(\epsilon)=0,\: \E(\|\epsilon\|^2)=1,
\end{equation*}
where $\sigma$ is the noise level. We suppose that $\E(\|Y\|^2)<+\infini$, and we define the pre-variance of $Y$ in $M/G$ as the map defined by:
\begin{equation*}
    F(m)=\E\left(\inf_{g\in G} \|g\cdot m-Y\|^2\right).
\end{equation*}

In this part we still study the inconsistency of template estimation by minimizing $F$.

We present two frameworks where we can ensure the presence of inconsistency: in \Cref{sec:subgroup} we suppose that the group $G$ contains a non trivial group $H$ which acts isometrically on $M$.
However, some groups do not satisfy this hypothesis, that is why, in \Cref{sec:linear} we do not suppose that $G$ contains a subgroup acting isometrically but we require that $G$ acts linearly on $M$. In both sections we prove inconsistency as soon as the variance $\sigma^2$ is large enough. 

These hypothesis are not unacceptable as for example, deformations that are considered in computational anatomy may include rotations which form a subgroup $H$ of the diffeomorphic deformations which acts isometrically. Concerning the second case, an important example is:

\begin{Example}
\label{exdiff}
Let $G$ be a subgroup of the group of $C^\infini$ diffeomorphisms on $\R^n$ $G$ acts linearly on $L^2(\R^n)$ with the map:
\begin{equation*}
\forall \varphi\in G\quad \forall f\in L^2(\R^n)\qquad \varphi\cdot f=f\circ\varphi^{-1}.
\end{equation*}

Note that this action is not isometric: indeed, $f\circ \varphi^{-1}$ has generally a different $L^2$-norm than $f$, because a Jacobian determinant appears in the computation of the integral.
\end{Example}

\subsection{Where Did We Need an Isometric Action Previously?}
Let $M$ be an Hilbert space, and $G$ a group acting on $M$. 
Can we define a distance in the quotient space $Q=M/G$ defined as the set which contains all the orbits? When the action is invariant, the~orbits are parallel in the sense where $d_M(m,n)=d_M(g\cdot m,g\cdot n)$ for all $m,n\in M$ and for all $g\in G$.
This~implies that:
\begin{equation*}
d_Q([m],[n])=\underset{g\in G}{\inf} \|m-g\cdot n\|,
\end{equation*}
is a distance on $Q$. However, it is not necessarily the case when the action is no longer invariant. Let us take the following example:
\begin{Example}
We call $C^\infini_{\mbox{diff}}(\R^2)$ the set of the $C^\infini$ diffeomorphisms of $\R^2$. We equip $\R^2$ with its canonical Euclidean structure. We take  $p=(-1,-1)$, $q=(1,1)$ and $r=(2,0)$ (see \cref{tio}),
\begin{equation}
G=\left\{ f\in C^\infini_{\mbox{diff}}(\R^2) \: | \: f(q)=(q),\: f(p)=(p),\: \forall x\in \R\: f(x,0)\in \R r\right\},
\end{equation}
$G$ acts on $\R^2$ by $f\cdot (x,y)=f(x,y)$. Then $q$ and $p$ are fixed points under this group action and the orbit of $r$ is the horizontal line $\{(x,0),x\in \R\}$.
On this example: 
\begin{equation*}\underset{g\in G}{\inf} \|q-g\cdot r\|=\|q-(1,0)\|=1\qquad
\mbox{ however } \qquad \underset{g\in G}{\inf} \|r-g\cdot q\|=\|r-q\|=\sqrt2,
\end{equation*}
then the function $d_Q$ is not symmetric. 
One could think define a distance by:
\begin{equation*}
\tilde d_Q([m],[n])=\underset{h,g\in G}{\inf} \|h\cdot m-g\cdot n\|.
\end{equation*}

Unfortunately, in this case $\tilde d_Q([p],[q])=\|p-q\|=2\sqrt 2$ and $\tilde d_Q([p],[r])=1=\tilde d_Q([r],[q])$ then we do not have $\tilde d_Q([p],[q])\leq \tilde d_Q([p],[r])+\tilde d_Q([r],[q])$. In other words we do not have the triangular inequality.
\begin{figure}[h]
\centering
\begin{tikzpicture}[scale=1.1]

\draw (-1,-1) node[below]{$p$} ;
\draw (1,1) node[above]{$q$} ;
\draw (2,0) node[above]{$r$} ;
\draw (-4,0) -- (3,0) ;
\draw[dotted,red, thick] (-1,0)--(-1,-1);
\draw[dotted,green, thick] (-1,-1)--(1,1);
\draw[dotted,blue,thick] (1,1)--(1,0);

\draw (-3,0) node[above]{$[r]$} ;

\draw (-1,-1) node {$\bullet$} ;
\draw (1,1) node {$\bullet$} ;
\draw (2,0) node {$\bullet$} ;

\draw[red] (-1,-0.5) node[left]{$\tilde d_Q([p],[r])=1$} ;
\draw[blue] (1,0.5) node[right]{$\tilde d_Q([q],[r])=1$} ;
\draw[green] (-0.4,-0.5) node[right]{$\tilde d_Q([p],[q])=2\sqrt2$};

\end{tikzpicture}

\caption{Example of three orbits, when $\tilde d_Q$ does not satisfy the inequality triangular.}
\label{tio} 
\end{figure}
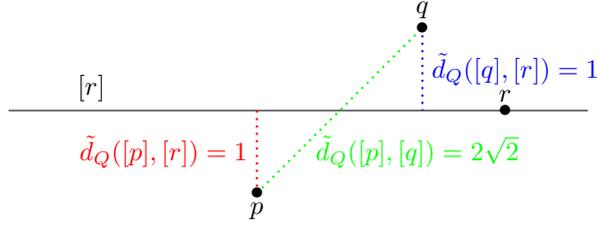
\end{Example}

Therefore when the action is no longer invariant, a priori one cannot define  a distance in the quotient anymore. If $Y$ is a random variable in $M$, $F(m)=\E(\inf_{g\in G} \|g\cdot m-Y\|^2)$ cannot be interpreted as the variance of $[Y]$.

However $\inf_{g\in G} \|g\cdot a-b\|$ is positive and is equal to zero if $a=b$, then $\inf_{g\in G}\|g\cdot a-b\|$ is a pre-distance in $M$. Then $\inf_{g\in G}\|g\cdot m-Y\|$ 
measures the discrepancy between the random point $Y$ and the current point $m$. Even if the discrepancy measure is not symmetric or does not satisfy the triangular inequality, we can still define $F(x)=\E(\inf_{g\in G} \|g\cdot x-Y\|^2)$ and call it the pre-variance of the projection of $Y$ into $M/G$, if $\E(\|Y\|^2)<+\infini$. The elements which minimize this function are the element whose orbit are the closest of the random point $Y$. Hence, we wonder if the template can be estimated by minimizing this pre-variance. Note that, once again $F(x)=F(g\cdot x)$ for all $x\in M$ and $g\in G$. Then the pre-variance is well defined in the quotient space by $[x]\mapsto F(x)$.

It is not surprising to use a discrepancy measure which is not a distance, for instance the Kullback-Leibler divergence~\cite{kul} is not symmetric although it is commonly used.

In the proof of inconsistency of \Cref{theo:Ks}, we used that the action was isometric in order to simplify the expansion of the variance in Equation~\eqref{expand}:
\begin{align*}
    F(m)&=\E\left(\inf_{g\in G} \| m-g\cdot Y\|^2\right)\\
    &=\E\left(\inf_{g\in G} \left[\|m\|^2 -\psh{m}{g\cdot Y} +\|g\cdot Y\|^2\right] \right),
\end{align*}
{with $\|g\cdot Y\|^2=\|Y\|^2$ there was only one term which depends on $g$: $\psh{g\cdot m}{Y}$ and the two other terms could be pulled out of the infimum. When the action is no longer isometric we cannot do this trick anymore.} To remedy this situation, in this article, we require that the orbit of the template is a bounded set.

{In the following, we prove inconsistency even with non isometric action (but only when the noise level is large enough if the template is not a fixed point). The sketches of the different proofs are always the same: finding a point $m$ such that $F(m)<F(t_0)$, in order to do that it suffices to find an upper bound of $F(m)$ and a lower bound of $F(t_0)$ and to compare these two bounds.}

\subsection{Non Invariant Group Action, with a Subgroup Acting Isometrically}
\label{sec:subgroup}

In this subsection $G$ acts on $M$ an Hilbert space. We assume that there exists a subgroup $H\subset G$ such that $H$ acts isometrically on $M$. As $H$ is included in $G$, we deduce a useful link between the variance of $Y$ projected in $M/H$ and the pre-variance of $Y$ projected in $M/G$:
\begin{equation*}
F(m)=\E(\underset{g\in G}{\inf} \|g\cdot m-Y\|^2)\leq \E(\underset{h\in H}{\inf} \|h\cdot m-Y\|^2)=F_H(m).
\end{equation*}

The orbit of a point $m$ under the group action $G$ is $[m]=\{g\cdot m,\: g\in G\}$, whereas the orbit of the point $m$ under the group action $H$ is $[m]_H=\{h\cdot m,\: h\in H\}$. Moreover, we call $F_H$ the variance of $[Y]_H$ in the quotient space $M/H$, and $F$ the variance of $[Y]$ in the quotient space $M/G$.

\subsubsection{Inconsistency when the Template Is a Fixed Point}
We begin by assuming that the template $t_0$ is a fixed point under the action of $G$:
\begin{Proposition}
Suppose that $t_0$ is a fixed point under the group action $G$. 
Let $\epsilon$ be a standardized noise which support is not included in the fixed points under the group action of $H$, and $Y=\Phi\cdot t_0+\sigma\epsilon=t_0+\sigma\epsilon$.
Then~$t_0$ is not a minimum of the pre-variance $F$.
\end{Proposition}
\begin{proof}
We have:
\begin{enumerate}
\item Thanks to \Cref{prop:fixedpoint} of \Cref{sec:inc} we know that $[t_0]_H=[\E(Y)]_H$ is not the Fréchet mean of $[Y]_H$ the projection of $Y$ into $M/H$:
we can find $m\in M$ such that:
\begin{equation}
F_{H}(m)<F_{H}(t_0).
\label{inegvar1}
\end{equation}
Note that in order to apply \Cref{prop:fixedpoint}, we do not need that $\Phi$ is included in $H$, because $t_0$ is a fixed point.
\item Because we take the infimum over more elements we have:
\begin{equation}
F(m)\leq F_{H}(m).
\label{inegvar2}
\end{equation}

\item As $t_0$ is a fixed point under the action of $G$ and under the action of $H$:
\begin{equation}
F_{H}(t_0)=F(t_0)=\E(\|t_0-Y\|^2).
\label{egal3}
\end{equation}
\end{enumerate}

With Equations~\eqref{inegvar1}--\eqref{egal3}, we conclude that $t_0$ does not minimize $F$.
\end{proof}

\subsubsection{Inconsistency in the General Case for the Template}
\label{subsec:gct}
The following \Cref{prop:sc} tells us that when $\sigma$ is large enough then there is an inconsistency.

\begin{Proposition}
\label{prop:sc}
We suppose that the template is not a fixed point and that its orbit under the group $G$ is bounded. We consider $A\geq \underset{g\in G}{\sup} \frac{\|g\cdot t_0\|}{\|t_0\|}$ and $a\leq \underset{g\in G}{\inf} \frac{\|g\cdot t_0\|}{\|t_0\|}$, note that $a\leq 1\leq A$ and we have:
\begin{equation*}
\forall g\in G\quad a\|t_0\|\leq \|g\cdot t_0\|\leq A\|t_0\|.
\end{equation*}

We note:
\begin{equation*}
\theta(t_0)=\frac{1}{\|t_0\|}\E(\sup_{g\in G}\psh{g\cdot t_0}{\epsilon}) \mbox{ and } \theta_H=\frac{1}{\|t_0\|}\E\left(\underset{h\in H}{\sup} \psh{h\cdot t_0}{\epsilon}\right).
\end{equation*}

We suppose that $\theta_H>0$. If $\sigma$ is bigger than a critical noise level noted $\sigma_c$ defined as:
\begin{equation}
\sigma_c=\frac{\|t_0\|}{\theta_H}\left[\left(\frac{\theta(t_0)}{\theta_H}+A\right)+\sqrt{\left(\frac{\theta(t_0)}{\theta_H}+A\right)^2+A^2-a^2}\right].
\label{sigmac}
\end{equation}

Then we have inconsistency.
\end{Proposition}
Note that in \Cref{sec:inc} we have proved inconsistency in the isometric case as soon as \mbox{$\sigma> \frac{2\|t_0\|}{K}$}, where $K\geq \theta_H$, then we find in this theorem an analogical sufficient condition on $\sigma$ where $\left[\left(\frac{\theta(t_0)}{\theta_H}+A\right)+\sqrt{\left(\frac{\theta(t_0)}{\theta_H}+A\right)^2+A^2-a^2}\right]$ is a corrective term due to the non invariant action.

{We have shown in~\cite{dev2} that if the orbit of the template $[t_0]_H$ is a manifold, then $\theta_H>0$ as soon as the support of $\epsilon$ is not included in $T_{t_0} [t_0]^\perp$ (the normal space of the orbit of the template $t_0$ at the point $t_0$). If~$[t_0]$ is not a manifold, we have also seen in~\cite{dev2} that $\theta_H>0$ as soon as $t_0$ is an accumulation point of $[t_0]_H$ and the support of $\epsilon$ contains a ball $B(0,r)$.} Hence, $\theta_H>0$ is a rather generic condition.
Condition~\eqref{sigmac} can be reformulated as follows: as soon as the signal to noise ratio $\frac{\|t_0\|}{\sigma}$ is sufficiently small:
\begin{equation*}
\frac{\|t_0\|}{\sigma}< \frac{\theta_H}{\left(\frac{\theta(t_0)}{\theta_H}+A\right)+\sqrt{\left(\frac{\theta(t_0)}{\theta_H}+A\right)^2+A^2-a^2}},
\end{equation*}
then there is inconsistency.

We remark the presence of the constants $\theta(t_0)$ and $\theta_H$ in \Cref{prop:sc}. This kind of constants were already here in the isometric case under the form $\theta(\frac{t_0}{\|t_0\|})=\frac{1}{\|t_0\|}\E(\underset{g\in G}{\sup} \psh{t_0}{g\cdot \epsilon})$, due to the polarization identity~\eqref{polari}, we can state that it measures how much the template looks like to the noise after registration, but only in the isometric case. However we can intuit that this constant plays a analogical role in the non isometric case.

\begin{Example}
Let $G$ acting on $M$, we suppose that $G$ contains $H=O(M)$ the orthogonal group of $M$. Assume~that $G$ can modifying the norm of the template by multiplying its norm by at most $2$. Then we can set up $A=2$ and $a=0$. By aligning $\epsilon$ and $\|t_0\|$ we have $\theta_H=\E(\|\epsilon\|)>0$, and $\theta(t_0)=A\E(\|\epsilon\|)$
then when the signal to noise ratio $\frac{\|t_0\|}{\sigma}$ is smaller that $\frac{\E(\|\epsilon\|)}{4+\sqrt{20}}$ then there is inconsistency. By Cauchy-Schwarz inequality we have $\E(\|\epsilon\|)\leq \E(\|\epsilon\|^2)=1$, thus the signal to noise ratio has to be rather small in order to fulfill this condition.
\end{Example}

\subsubsection{Proof of \Cref{prop:sc}}

We define the following values:
\begin{equation*}
    \lambda_H=\frac{1}{\|t_0\|^2} \E\left(\underset{h\in H}{\sup} \psh{h\cdot t_0}{Y}\right)
\mbox{ and }
    \lambda(t_0)=\frac{1}{\|t_0\|^2} \E\left(\underset{g\in G}{\sup} \psh{g\cdot t_0}{Y}\right).
\end{equation*}

Note that $\lambda_H$ and $\lambda(t_0)$ are registration scores which definitions are the same than the registration score used in the proof of \Cref{theo:Ks} in \Cref{sec:iso} (only the normalization by $\|t_0\|$ is different). The~proof of \Cref{prop:sc} is based on the following Lemma:
\begin{Lemma}
\label{theo:Ksg}
If:
\begin{align}
\lambda_H&\geq0, \label{lambdapos}\\
a^2-2\lambda(t_0)+\lambda_H^2&>0, \label{hypconsts}
\end{align}
then $t_0$ is not a minimizer of the pre-variance of $[Y]$ in $M/G$.
\end{Lemma}

How condition~\eqref{hypconsts} can be understood? In order to answer to that question, let us imagine that $G=H$ acts isometrically, then $a$ can be set up to $1$, and $\lambda(t_0)=\lambda_H$ the condition~\eqref{hypconsts} becomes $\lambda_H^2-2\lambda_H+1=(\lambda_H-1)^2>0$ and the conditions of Theorem~4.2 of~\cite{dev2} aimed to ensure that $\lambda_H>1$. Now let us return to the non invariant case: if $H$ is strictly included in $G$ such that $a$ is closed enough to $1$ and $\lambda(t_0)$ closed enough to $\lambda_H$, then on can think that condition~\eqref{hypconsts} still holds. However, the \textit{closed~enough} seems hard to be quantified. 
\begin{proof}[Proof of \Cref{theo:Ksg}]
The proof is based on the following points:
\begin{enumerate}
\item\label{FGFH} $F(\lambda_H t_0)\leq F_H(\lambda_H t_0)$,
\item\label{FGFH2} $F_H(\lambda_H t_0)< F( t_0)$.
\end{enumerate}

With \cref{FGFH,FGFH2} we get that $F(\lambda_H t_0)<F(t_0)$. \Cref{FGFH} is just based on the fact that in the map $F$, we take the infimum on a larger set than on $F_H$. We now prove \cref{FGFH2}, in order to do that we expand the two quantities, firstly:
\begin{align}
F_H(\lambda_H t_0)&=\E\left(\inf_{h\in H} \|h \cdot \lambda_H t_0\|^2+\|Y\|^2-2\psh{h\cdot \lambda_H t_0}{Y}\right) \label{iso1}\\
&= \lambda_H^2 \|t_0\|^2+\E(\|Y\|^2)-2\lambda_H \E\left(\underset{h\in H}{\sup} \psh{h\cdot t_0}{Y}\right)\label{iso2}\\
&=\E(\|Y\|^2)-\lambda_H^2\|t_0\|^2,\nonumber
\end{align}

We use the fact that $H$ acts isometrically between Equations~\eqref{iso1} and~\eqref{iso2} and the fact that $\lambda_H\geq 0$ because $\inf_{a\in A}-\lambda a=-\lambda \sup_{a\in A} a$ is true for any~$A$ subset of~$\R$ if $\lambda\geq 0$. Secondly:
\begin{align*}
F(t_0)&=\E\left(\inf_{g\in G} \|g\cdot t_0\|^2+\|Y\|^2-2\psh{g\cdot t_0}{Y}\right)\\
&\geq  a^2\|t_0\|^2+\E(\|Y\|^2)-2 \E\left(\underset{g\in G}{\sup} \psh{g\cdot t_0}{Y}\right)\\
&\geq a^2\|t_0\|^2+\E(\|Y\|^2)-2\lambda(t_0)\|t_0\|^2
\end{align*}

Then:
\begin{equation*}
F(t_0)-F_H(\lambda_H t_0)\geq \|t_0\|^2\left[a^2-2\lambda(t_0)+\lambda_H^2\right]>0,
\end{equation*}
thanks to hypothesis~\eqref{hypconsts}.
\end{proof}

\begin{proof}[Proof of \Cref{prop:sc}]
{In order to prove \Cref{prop:sc}, all we have to do is proving $\lambda_H\geq 0$ and proving that Condition~\eqref{hypconsts} is fulfilled when $\sigma>\sigma_c$}.
Firstly, thanks to Cauchy-Schwarz inequality, we have:
\begin{align*}
\lambda_H&=\frac{1}{\|t_0\|^2}\E\left(\underset{h\in H}{\sup} \psh{h\cdot t_0}{\Phi\cdot t_0+\sigma\epsilon}\right)\\
&\geq  \frac{1}{\|t_0\|^2} \left[-A \|t_0\|^2+\E(\underset{h\in H}{\sup} \psh{h\cdot t_0}{\sigma\epsilon})\right]\geq  -A+\sigma \frac{\theta_H}{\|t_0\|}
\end{align*}

Note that as $\sigma>\sigma_c\geq A \frac{\|t_0\|}{\theta_H}$ we get $\lambda_H\geq 0$, this proves~\eqref{lambdapos}. We also have:
\begin{align*}
\lambda(t_0)&= \frac{1}{\|t_0\|^2}  \E\left(\underset{g\in G}{\sup} \psh{g\cdot t_0}{\Phi\cdot t_0+\sigma\epsilon}\right)\\
&\leq \frac{1}{\|t_0\|^2} \left[A^2\|t_0\|^2+\sigma \E\left( \underset{g\in G}{\sup} \psh{g\cdot t_0}{\epsilon} \right)\right]\leq  A^2+\sigma \frac{\theta(t_0)}{\|t_0\|},
\end{align*}

Then we can find a lower bound of $a^2-2\lambda(t_0)+\lambda_H^2$:
\begin{align*}
a^2-2\lambda(t_0)+\lambda_H^2&\geq a-2\left(A^2+\sigma \frac{\theta(t_0)}{\|t_0\|}\right)+\left(\frac{\sigma \theta_H}{\|t_0\|}-A\right)^2 \\
&\geq  a^2-A^2-2\frac{\sigma\theta_H}{\|t_0\|}\left(\frac{\theta(t_0)}{\theta_H}+A\right)+\left(\frac{\sigma\theta_H}{\|t_0\|}\right)^2:=P(\sigma)
\end{align*}

For $\sigma>\sigma_c$ where $\sigma_c$ is the biggest solution of the quadratic Equation $P(\sigma)=0$, we get \mbox{$a^2-2\lambda(t_0)+\lambda_H^2>0$} and template estimation is inconsistent thanks to \Cref{theo:Ksg}. The critical $\sigma_c$ is exactly the one given by \Cref{prop:sc}.
\end{proof}

\subsection{Linear Action}
\label{sec:linear}
The result of the previous part has a drawback, it requires that the group of deformations contains a non trivial subgroup which acts isometrically. We know remove this hypothesis, but we require that the group acts linearly on data.
\subsubsection{Inconsistency}
In this Subsection we suppose that the group $G$ acts linearly on $M$. Once again, we can give a criteria on the noise level which leads to inconsistency:
\begin{Proposition}
\label{prop:lin}
We suppose that the orbit of the template is bounded with:
\begin{equation*}
\exists a\geq 0, A>0 \mbox{ such that } \forall g\in G\quad a\|t_0\|\leq \|g\cdot t_0\|\leq A\|t_0\|.
\end{equation*}

We suppose that $A<\sqrt 2$. In other words, the deformation of the template can multiply the norm of the template by less than $\sqrt 2$. We also suppose that:
\begin{equation}
\theta(t_0)=\frac{1}{\|t_0\|}\E\left(\underset{g\in G}{\sup} \psh{g\cdot t_0}{\epsilon}\right)>0.
\label{nunoniso}
\end{equation}

There is inconsistency as soon as
\begin{equation*}
\sigma\geq \sigma_c=\frac{\|t_0\|}{\theta(t_0)}\left[A^2+\frac{1+\sqrt{1-a^2(2-A^2)}}{2-A^2}\right].
\end{equation*}
\end{Proposition}

\begin{Example}
For instance if $A\leq 1.2$, then there is inconsistency if $\sigma\geq 7\frac{\|t_0\|}{\theta(t_0)}$.
\end{Example}
Once again we find a condition which is similar to the isometric case, but due to the non invariant action we have here a corrective term which depends on $A$ and $a$.
Note that as $G$ does not act isometrically, results in~\cite{dev2} do not apply in order to fulfill Condition~\eqref{nunoniso}. However it is easy to fulfill this Condition thanks to the following Proposition:

\begin{Proposition}
\label{prop:nug}
If $t_0$ is not a fixed point, and if the support of $\epsilon$ contains a ball $B(0,\rho)$ for $\rho>0$ then
\begin{equation*}
\theta(t_0)=\frac{1}{\|t_0\|}\E\left(\underset{g\in G}{\sup} \psh{g\cdot t_0}{\epsilon}\right)>0.
\end{equation*}
\end{Proposition}

\begin{Remark}
It is possible to remove the condition $A<\sqrt 2$ in \Cref{prop:lin}. Indeed Let be $h\in G$ such that:

\begin{equation*}
\frac{\underset{g\in G}{\sup} \|g\cdot t_0\|}{\|h\cdot t_0\|}<\sqrt 2.
\end{equation*}

The template $t_0$ can be replaced by $h\cdot t_0$ since $\Phi t_0+\sigma \epsilon$ is equal to $\Phi h^{-1} \cdot ht_0$ and applying \Cref{prop:lin} to the new template $h\cdot t_0$. We get that $h\cdot t_0$ does not minimize the variance $F$ with $A\leq \sqrt 2$ (because the new template is $h\cdot t_0)$. 
Since $h\cdot t_0$ does not minimize $F$, the original template $t_0$ does not minimize the pre-variance $F$ neither, since $F(t_0)=F(h\cdot t_0)$.

This changes the critical $\sigma_c$ since we apply \Cref{prop:lin} to $h\cdot t_0$ instead of $t_0$ itself.
\end{Remark}

\subsubsection{Proofs of \Cref{prop:lin} and \Cref{prop:nug}}
As in \Cref{sec:subgroup} we first prove a Lemma:
\begin{Lemma}
\label{theo:lin}
We define:
\begin{equation*}
\lambda(t_0)=\frac{1}{\|t_0\|^2}  \E\left(\underset{g\in G}{\sup} \psh{g\cdot t_0}{Y}\right).
\end{equation*}

Suppose that $\lambda(t_0)\geq 0$ and that: 
\begin{equation}
a^2-2\lambda(t_0) +\lambda(t_0)^2(2-A^2)>0.
\label{condlin}
\end{equation}

Then $t_0$ is not a minimum of $F$.
\end{Lemma}

\begin{proof}[Proof of \Cref{theo:lin}]
Since
\begin{equation}
 \forall g\in G  \quad a\|t_0\|\leq \|g\cdot t_0\|\leq A\|t_0\|,
 \label{minmajt}
\end{equation}
then by linearity of the action we get:
\begin{equation}
\forall g\in G,\: \mu\in \R\quad a\|\mu t_0\|\leq \|g\cdot \mu t_0\|\leq A\|\mu t_0\|.
\label{minmajmut}
\end{equation}

We remind that:
\begin{equation*}
F(m)=\E\left( \underset{g\in G}{\inf} \| g\cdot m\|^2 -2\psh{g\cdot m}{Y} +\|Y\|^2 \right).
\end{equation*}

By using Equations~\eqref{minmajt} and~\eqref{minmajmut} we get:
\begin{equation*}
F(t_0)\geq a^2\|t_0\|^2-2\lambda(t_0) \|t_0\|^2+\E(\|Y\|^2),
\end{equation*}

We get:
\begin{align}
F(\lambda(t_0) t_0)&\leq \E\left( A^2\|\lambda(t_0) t_0\|^2+\|Y\|^2+\underset{g\in G}{\inf}-2\lambda(t_0)\psh{g\cdot  t_0}{Y}\right)\label{lineq}\\
 &\leq A^2\lambda(t_0)^2\|t_0\|^2 +\E(\|Y\|^2) -2\lambda(t_0)^2\|t_0\|^2.\nonumber
\end{align}

Note that we use the fact that the action is linear in Equation~\eqref{lineq}. We obtain that $t_0$ is not the minimum of the $F$:
\begin{equation*}
F(t_0)-F(\lambda(t_0) t_0)\geq \|t_0\|^2\left[a^2-2\lambda(t_0) +\lambda(t_0)^2(2-A^2)\right]>0.
\end{equation*}
\end{proof}
\begin{proof}[Proof of \Cref{prop:lin}]
By solving the following quadratic inequality we remark that:
\begin{equation*}
a^2-2\lambda(t_0)+(2-A^2)\lambda(t_0)^2>0 \mbox{ if } \lambda(t_0)>\frac{1+\sqrt{1-a^2(2-A^2)}}{2-A^2},
\end{equation*}

Besides, as in \cref{subsec:gct} we can take a lower bound of $\lambda(t_0)$ by decomposing $Y=\Phi\cdot t_0+\sigma \epsilon$ and applying Cauchy-Schwarz inequality $\psh{\Phi\cdot t_0}{g\cdot t_0}\geq -A^2\|t_0\|^2$, we get:

\begin{equation}
\lambda(t_0)\geq -A^2+\frac{\sigma}{\|t_0\|}\theta(t_0).
\label{nugl}
\end{equation}

Thanks to Condition~\eqref{nugl} and the fact that $\sigma>\sigma_c$
we get:
\begin{equation*}
\lambda(t_0)\geq -A^2+\frac{\sigma}{\|t_0\|}\theta(t_0)> \frac{1+\sqrt{1-a^2(2-A^2)}}{2-A^2}
\end{equation*}

Then $\lambda(t_0)\geq 0$ and Condition~\eqref{condlin} is fulfilled. Thus, there is inconsistency, according to \Cref{theo:lin}.
\end{proof}

\begin{proof}[Proof of \Cref{prop:nug}]
First we notice that:
\begin{equation}
\|t_0\|\theta(t_0)=\E\left(\underset{g\in G}{\sup} \psh{g\cdot t_0}{\epsilon}\right)\geq \E(\psh{t_0}{\epsilon})=\psh{t_0}{\E(\epsilon)}= 0.
\label{nugpositive}
\end{equation}

In order to have $\theta(t_0)>0$, first we show that it exists $x\in B(0,\rho)$ and $g_0\in G$ such that
\begin{equation*}
\underset{g\in G}{\sup} \psh{g\cdot t_0}{x}\geq \psh{g_0\cdot t_0}{x}>\psh{t_0}{x}.
\end{equation*}

Let $g_0\in G$ such that $g_0\cdot t_0\neq t_0$. There are three cases to be distinguished (see \cref{fig:3cases}):

\begin{enumerate}
\item The vectors $g_0\cdot t_0$ and $t_0$ are linearly independent. In this case $t_0^\perp \not\subset (g_0\cdot t_0)^\perp$, then we can find $x\in t_0^\perp$ and $x\notin(g\cdot t_0)^\perp$. Then $\psh{t_0}{x}=0$ and $\psh{g\cdot t_0}{x}\neq 0$, without loss of generality we can assume that $\psh{g\cdot t_0}{x}>0$ (replacing $x$ by $-x$ if necessary). We also can assume that $x\in B(0,\rho)$ (replacing $x$ by $\frac{x\rho}{2\|x\|}$ if necessary. Then we have $x\in B(0,\rho)$ and:
\begin{equation*}
\psh{g_0\cdot t_0}{x}>0=\psh{t_0}{x}.
\end{equation*} 
\item If $g_0\cdot t_0=w t_0$ with $w>1$, we take $x=\frac{\rho}{2\|t_0\|}t_0\in B(0,\rho)$ and we have:
\begin{equation*}
\psh{g\cdot t_0}{x}=w\frac{\rho}{2}\|t_0\|> \frac{\rho}{2}\|t_0\|=\psh{t_0}{x}.
\end{equation*}
\item If $g_0\cdot t_0=wt_0$ with $w<1$ we take $x=-\frac{\rho}{2\|t_0\|}{t_0}\in B(0,\rho)$ and we have:
\begin{equation*}
    \psh{g_0\cdot t_0}{x}=-w\frac{\rho}{2}\|t_0\|> - \frac{\rho}{2}\|t_0\|=\psh{t_0}{x}.
\end{equation*}

\end{enumerate}

In all these cases we can find $x$ such that $\psh{g_0\cdot t_0}{x}>\psh{t_0}{x}$ 
By continuity it exists $r>0$ such that for all $y$ on this ball we have $\psh{g \cdot t_0}{y}> \psh{t_0}{y}$. Then the event $\{\sup_{g\in G} \psh{g \cdot t_0}{\epsilon}> \psh{t_0}{\epsilon}\}$ has non zero probability, since $x$ is in the support of $\epsilon$ we have $\P(\epsilon\in B(x,r))>0$. Thus Inequality in~\eqref{nugpositive} will be strict. This proves that $\theta(t_0)>0$. 

\end{proof}
\vspace{-12pt}
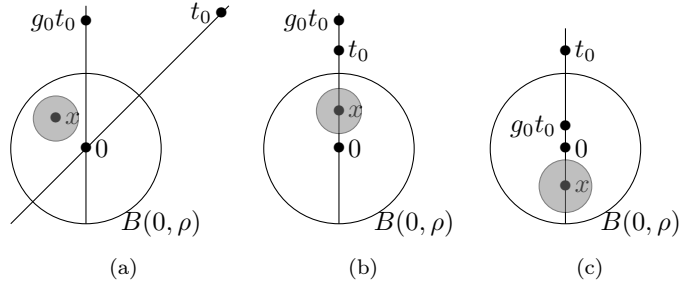
\begin{figure}[h]
\centering
\subfloat[] 
{\begin{tikzpicture}
\draw (0,0) node{$\bullet$} node[right] {$0$};
\draw (-1,-1)--(1.9,1.9);
\draw (0,-1)--(0,1.9);
\draw (0,1.7) node{$\bullet$} node[left] {$g_0 t_0$};
\draw (1.8,1.8) node{$\bullet$} node[left] {$t_0$};
\draw (-0.4,0.4) node{$\bullet$} node[right] {$x$};
\draw (1,-0.7) node[below] {$B(0,\rho)$};
\draw[fill=gray, opacity=0.5] (-0.4,0.4) circle(0.3);
\draw (0,0) circle(1);
\end{tikzpicture}}
\quad
\subfloat[] 
{\begin{tikzpicture}
\draw (0,0) node{$\bullet$} node[right] {$0$};
\draw (0,-1)--(0,1.8);
\draw (0,1.7) node{$\bullet$} node[left] {$g_0t_0$};
\draw (0,1.3) node{$\bullet$} node[right] {$t_0$};
\draw (0,0.5) node{$\bullet$} node[right] {$x$};
\draw (1,-0.7) node[below] {$B(0,\rho)$};
\draw[fill=gray, opacity=0.5] (0,0.5) circle(0.3);
\draw (0,0) circle(1);
\end{tikzpicture}}
\quad
\subfloat[] 
{\begin{tikzpicture}
\draw (0,0) node{$\bullet$} node[right] {$0$};
\draw (0,-1)--(0,1.6);
\draw (0,0.3) node{$\bullet$} node[left] {$g_0t_0$};
\draw (0,1.3) node{$\bullet$} node[right] {$t_0$};
\draw (0,-0.5) node{$\bullet$} node[right] {$x$};
\draw (1,-0.7) node[below] {$B(0,\rho)$};
\draw[fill=gray, opacity=0.5] (0,-0.5) circle(0.35);
\draw (0,0) circle(1);
\end{tikzpicture}}
\caption[Three cases in order to show that $\theta(t_0)>0$.]{Representation of the three cases, on each we can find an $x$ in the support of the noise such as $\psh{x}{g_0\cdot t_0}>\psh{x}{t_0}$ and by continuity of the dot product $\psh{\epsilon}{g_0\cdot t_0}>\psh{\epsilon}{t_0}$ with is an event with a non zero probability, (for instance the ball in gray). This is enough in order to show that $\theta(t_0)>0$. (\textbf{a})~Case 1: $t_0$ and $g\cdot t_0$ are linearly independent; (\textbf{b}) Case 2: $g\cdot t_0$ is proportional to $t_0$ with a factor $>1$;  (\textbf{c})~Case 3: $g\cdot t_0$ is proportional to $t_0$ with a factor $<1$.} 
\label{fig:3cases}
\end{figure}

\subsection{Example of a Template Estimation Which is Consistent}
In order to underline the importance of the hypotheses, we give an example where the method is~consistent:
\begin{Example}
\label{exaffine}
Let $M$ be an Hilbert space and $V$ a closed sub-linear space of $M$. Then $G=V$ acts on $M$ by (see \cref{fig:aff}):
\begin{equation*}
(v,m)\in G\times M\mapsto m+v.
\end{equation*}

This action is not isometric, indeed $m\mapsto m+v$ is not linear (except if $v=0$). However this action is invariant, let us consider $V^\perp$ the orthogonal space of $V$. The variance in the quotient space is:
\begin{equation*}
F(m)=\E\left(\underset{v\in V}{\inf} \|m+v-Y\|^2\right)=\E(\|p(m)-p(Y)\|^2)=\E(\|p(m)-p(t_0)+\epsilon\|^2),
\end{equation*}
where $p:M\to V^\perp$ the orthogonal projection on $V^\perp$. Then it is clear that $t_0$ minimizes $F$. In fact, \mbox{$s:[m]\mapsto p(m)$} is just a congruent section of the quotient (see \Cref{subsec:equisec}). Here, once again, we see the role played by the the congruent section (when it exists) in order to study the consistency.
\end{Example}

Hence, is there a contradiction with \Cref{prop:sc} or \Cref{prop:lin} which prove inconsistency as soon as the noise level is large enough? In \Cref{prop:sc}, we require that there is a subgroup acting isometrically, in this example the only element which acts linearly is the identity element $m\mapsto m+0$, then $H=\{0\}$ is the only possibility, however the support of the noise should not be included in the set of fixed point under the group action of $H$. Here, all points are fixed under $H$, hence it is not possible to fulfill this condition. \Cref{exaffine} is not a contradiction with \Cref{prop:sc}, it is also not a contradiction with \Cref{prop:lin} since it does not act linearly on data.

\begin{figure}[h]
\centering
\begin{tikzpicture}
\draw (0,-1)--(0,2);
\draw (-1,0)--(7,0);
\draw (7,0) node[below] {$V^\perp$};
\draw (0,2) node[left] {$V$};
\draw (1,1) node {$\bullet$} node[left] {$x$};
\draw (5,0.5) node {$\bullet$} node[right] {$y$};
\draw (1,0) node {$\bullet$} node[below left] {$p(x)$};
\draw (5,0) node {$\bullet$} node[below right] {$p(y)$};
\draw[thick] (5,0)--(1,0);
\draw (3,0) node[above] {$d_Q([x],[y])$};
\draw (1,-1)--(1,2);
\draw (5,-1)--(5,2);
\draw (1,2) node[left] {$[x]$};
\draw (5,2) node[right] {$[y]$};
\end{tikzpicture}
\caption[Affine translation by a subspace of $M$.]{In the case of affine translation by vectors of $V$, the orbits are affine subspace parallel to $V$. The distance between two orbits $[x]$ and $[y]$ is given by the distance between the orthogonal projection of $x$ and $y$ in $V^\perp$. This is an example where template estimation is consistent.} 
\label{fig:aff}
\end{figure}
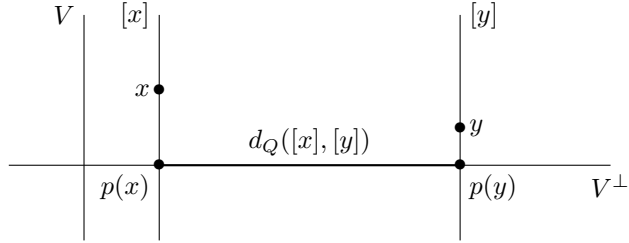

\subsection{Inconsistency with Non Invariant Action and Regularization}
In practice people add a regularization term in the function they minimize in LDDMM~\cite{mil,dur}, or~in~Demons~\cite{lom} etc.
Because, if one considers two points, one does not want necessarily to fit one with the other. Indeed, even if one deformation matches exactly these two points, it may be an unrealistic deformation.
So far, we did not study the use of such a term in the inconsistency.
\subsubsection{Case of Deformations Closed to the Identity Element of $G$}
If we suppose that the deformations $\Phi$ of the template is closed to identity, it is useless to take the infimum over $G$ because $G$ contains big deformations. 
Perhaps one of these big deformations can reaches the infimum in $F$, but this element is not the one which deformed the template in the generative model. Then such big deformations should not be taken into account. That is why, if we suppose that $G$ can be equipped with a distance $d_G$, then we can assume that there exists $r>0$ such that the deformation $\Phi$ belongs almost surely to 
\begin{equation*}
\Bc=B(e,r)=\{g\in G,\quad d_G(e,g)<r\}.
\end{equation*} 

Instead of defining $F(m)$ as $\E(\inf_{g\in G}\|g\cdot m-Y\|^2)$, one can define \mbox{$F(m)=\E(\inf_{g\in \Bc}\|g\cdot m-Y\|^2)$}, and the previous proofs will still be true, when replacing for instance $\lambda(t_0)$ by\mbox{ $\lambda(t_0)=\frac1{\|t_0\|^2}\E(\sup_{g\in \Bc} \psh{g\cdot t_0}{Y})$} etc. Likewise we need to replace the hypothesis ``the support of $\epsilon$ is not included in the set of fixed points `` by ''the support of $\epsilon$ in not included is the set of fixed points under the action restricted to $\Bc$''.

{Note that restraining ourselves to $\Bc$ is equivalent to add a following regularization on the function $F$}:
\begin{equation*}
    F(m)=\E\left(\underset{g\in G}{\inf} \|g\cdot m-Y\|^2 +Reg(g) \right)\mbox{ with } Reg(g)=\left\{ \begin{array}{ccc}
       0  & \mbox{ if} & g\in \Bc \\
       +\infini  & \mbox{ if } & g\notin \Bc
    \end{array}\right. .
\end{equation*}

Moreover considering only the elements in $\Bc$ will automatically satisfy the condition $A<\sqrt 2$ in \Cref{prop:lin} as long as the group $G$ acts continuously on the template, if $r$ is small enough.

\subsubsection{Inconsistency in the Case of a Group Acting Linearly with a Bounded Regularization}

In this Section we suppose that the group $G$ acts linearly. We also suppose that $A<\sqrt 2$. The~regularization term is a bounded map $Reg:G\to [0,\Omega]$. With this framework, we still able to prove that there is inconsistency as soon as the noise level is large enough:
\begin{Proposition}
Let $G$ be a group acting linearly on $M$. We suppose that the orbit of the template $t_0$ is bounded with $A=\underset{g\in G}{\sup} \frac{\|g\cdot t_0\|}{\|t_0\|}<\sqrt 2$, the generative model is still $Y=\Phi\cdot t_0+\sigma \epsilon$. We define the pre-variance as:
\begin{equation*}
F(m)=\E \left( \underset{g\in G}{\inf} \left( \|Y-g\cdot m\|^2+Reg(g)\right)\right).
\end{equation*}

Then as soon as the noise level is large enough, i.e.,:
\begin{equation*}
\sigma> \sigma_c=\frac{\|t_0\|}{\theta(t_0)}\left[A^2+\frac{1+\sqrt{1-(a^2+\frac{\Omega}{\|t_0\|^2})(2-A^2)}}{2-A^2}\right].
\end{equation*}

Then $t_0$ is not a minimizser of $F$.
\end{Proposition}
The proof is exactly the same as the Proof of \Cref{prop:lin}, we take $0$ as a lower bound of the the regularization term in the lower bound of $F(t_0)$, and we take $\Omega$ as a upper bound of the regularization term in the upper bound of $F(\lambda(t_0) t_0)$. We solve a similar quadratic equation in order to find the critical $\sigma$.

\section{Conclusions and Discussion}
We provided an asymptotic behavior of the consistency bias when the noise level $\sigma$ tends to infinity in the case of isometric action. As a consequence, the inconsistency can not be neglected when $\sigma$ is large. When the action is no longer isometric, inconsistency has been also shown when the noise level is large.

However, we have not answered this question: can the inconsistency be neglected? When the noise level is small enough, then the consistency bias is small~\cite{mio2,dev2}, hence it can be neglected. Note~that the quotient space is not a manifold, this prevents us to use a priori the Central Limit theorem for manifold proved in~\cite{bha}. However, if the Central Limit theorem could be applied to quotient space, the fluctuations induces an error which would be approximately equal to $\frac{\sigma}{\sqrt I}$ and if $K\ll\frac{1}{\sqrt I}$, then the inconsistency could be neglected because it is small compared to fluctuation. One way to avoid the inconsistency is to use another framework, for a instance a Bayesian paradigm~\cite{che}.

In the numerical experiments we presented, we have seen that the estimated template is more crispy that the true template. The intuition is that the estimated template in computational anatomy with a group of diffeomorphisms is also more detailed. However, the true template is almost always unknown. It is then possible that one think that the computation of the template succeeded to capture small details of the template while it is just an artifact due to the inconsistency. Moreover in order to tackle this question, one needs to have a good modeliation of the noise, for instance in~\cite{kur}, the observations are curves, what is a relevant noise in the space of curves?

In this article, we have considered actions which do not let the distance invariant. Although we have only shown the inconsistency as soon as the noise level is large enough, the inequality used where not optimal at all, surely future works could improve this work and prove that inconsistency appears for small noise level. Moreover a quantification of the inconsistency should be established.

\bibliographystyle{alpha}
\bibliography{bi}

\end{document}